\documentclass[a4paper]{article}

\usepackage[dvipsnames]{xcolor}
\usepackage[english]{babel}
\usepackage{amsfonts}
\usepackage[T1]{fontenc}
\usepackage{algorithm}
\usepackage[noend]{algpseudocode}

\usepackage[a4paper,top=3cm,bottom=2cm,left=3cm,right=3cm,marginparwidth=1.75cm]{geometry}
\usepackage{tikz}
\usepackage[utf8]{inputenc}
\usepackage{pgfplots} 
\usepackage{pgfgantt}
\usepackage{pdflscape}
\pgfplotsset{compat=newest} 
\pgfplotsset{plot coordinates/math parser=false} 
\newlength\fwidth
\newlength\fheight
\usepackage{amsmath}
\usepackage{amsthm}
\usepackage{graphicx}
\usepackage{subcaption}
\usepackage[textsize=tiny]{todonotes}

\usepackage[colorlinks=true, allcolors=blue]{hyperref}
\newtheorem{theorem}{Theorem}
\theoremstyle{definition}
\newtheorem{remark}{Remark}
\newtheorem{example}{Example}
\newtheorem{assumption}{Assumption}
\newcommand{\Real}{\mathbb{R}}
\newcommand{\bzero}{\mathbf{0}}
\newcommand{\cH}{\mathcal{H}}
\newcommand{\dH}{ H_{\rm d}}
\newcommand{\dx}{\Delta x}
\newcommand{\y}{\mathbf y}
\newcommand{\q}{\mathbf q}
\newcommand{\p}{\mathbf p}
\newcommand{\qhat}{\hat{\mathbf q}}
\newcommand{\phat}{\hat{\mathbf p}}
\newcommand{\Y}{\mathbf Y}
\newcommand{\Q}{\mathbf Q}
\newcommand{\Qhat}{\hat{\mathbf Q}}
\newcommand{\Phat}{\hat{\mathbf P}}

\newcommand{\Dhatq}{\hat{\mathbf D}_{\q}}
\newcommand{\Dhatp}{\hat{\mathbf D}_{\p}}
\newcommand{\Pp}{\mathbf P}
\newcommand{\V}{\mathbf V}
\renewcommand{\v}{\mathbf v}
\newcommand{\Vp}{\mathbf V_{\p}}
\newcommand{\Vq}{\mathbf V_{\q}}
\newcommand{\Hr}{\hat{ H}}
\newcommand{\Fq}{\mathbf F_{\q}}
\newcommand{\Fp}{\mathbf F_{\p}}
\newcommand{\Fhatq}{\hat{\mathbf F}_{\q}}
\newcommand{\Fhatp}{\hat{\mathbf F}_{\p}}
\newcommand{\In}{\mathbf I_n}
\newcommand{\fq}{\mathbf f_{\q}}
\newcommand{\fp}{\mathbf f_{\p}}
\newcommand{\Jn}{\mathbf J_{2n}}
\newcommand{\Jr}{\mathbf J_{2r}}
\newcommand{\dt}{\Delta t}
\newcommand{\Vt}{\V^{+}}

\title{Hamiltonian Operator Inference:\\ Physics-preserving Learning of Reduced-order Models \\ for Canonical Hamiltonian Systems}

\author{Harsh Sharma\thanks{Corresponding Author, hasharma@ucsd.edu} \thanks{Postdoctoral Scholar, Department of Mechanical and Aerospace Engineering, University of California San Diego}, Zhu Wang\thanks{Associate Professor, Department of Mathematics, University of South Carolina}  \ and Boris Kramer\thanks{Assistant Professor, Department of Mechanical and Aerospace Engineering, University of California San Diego}}
\begin{document}
\maketitle

\begin{abstract} 
This work presents a nonintrusive physics-preserving method to learn reduced-order models (ROMs) of canonical Hamiltonian systems. Traditional intrusive projection-based model reduction approaches utilize symplectic Galerkin projection to construct Hamiltonian ROMs by projecting  Hamilton's equations of the full model onto a symplectic subspace. This symplectic projection requires complete knowledge about the full model operators and full access to manipulate the computer code. In contrast, the proposed Hamiltonian operator inference approach embeds the physics into the operator inference framework to develop a data-driven model reduction method that preserves the underlying symplectic structure. Our method exploits knowledge of the Hamiltonian functional to define and parametrize a Hamiltonian ROM form which can then be learned from data projected via symplectic projectors. The proposed method is gray-box in that it utilizes knowledge of the Hamiltonian structure at the partial differential equation level, as well as knowledge of spatially local components in the system. However, it does not require access to computer code, only data to learn the models.  Our numerical results demonstrate Hamiltonian operator inference on a linear wave equation, the cubic nonlinear Schr\"{o}dinger equation, and a nonpolynomial sine-Gordon equation. Accurate long-time predictions far outside the training time interval for nonlinear examples illustrate the generalizability of our learned models. \\ \\
\textbf{Keywords:} Structure-preserving model reduction;  Hamiltonian systems; Physics-informed machine learning; Data-driven modeling; Operator inference.

\end{abstract}
%
%
%
%
\section{Introduction}
\label{s:1}
Hamiltonian partial differential equations (PDEs) arise as models in many science and engineering applications such as the elasticity equations in elastodynamics, the Maxwell-Vlasov equations in plasma physics, the shallow-water equations in climate modeling, and the Kuramoto–Sivashinsky equation in chemical reaction dynamics, see, e.g., \cite{marsden2013introduction}. The governing equations in Hamiltonian systems possess physical, mechanical and mathematical structures in the form of symmetries, symplecticity, Casimirs, and energy conservation. The conservative nature and the underlying symplectic structure of Hamiltonian systems are considered fundamental to their discretization and numerical treatment.
\par
In the last three decades, the field of geometric numerical integration has produced a variety of numerical methods for simulating physical systems described by Hamiltonian ordinary differential equations (ODEs), which respect the qualitative features of the dynamical system. These structure-preserving ideas have also been extended to Hamiltonian PDEs. An overview of the field of structure-preserving methods can be found in, e.g. \cite{bridges2006numerical, sharma2020review}. For a thorough exposition, the interested reader may consult the standard textbooks \cite{hairer2006geometric,leimkuhler_reich_2005} and the references cited herein. Since many applications of Hamiltonian systems involve long-time numerical simulations of large-scale systems, reduced-order models (ROMs) can be employed to obtain surrogate models that can be integrated in time at much lower computational cost. The qualitative properties of the surrogate model are critical to the accuracy of the numerical simulation and reliability of long-time predictions.
\par
Among the many model reduction approaches, proper orthogonal decomposition (POD) with Galerkin projection \cite{sirovich1987turbulence,berkooz1993proper,lumley1967structure,holmes2012turbulence} has proven beneficial in a variety of science and engineering applications. In projection-based model reduction, the governing equations are projected onto a low-dimensional subspace spanned by POD basis vectors. Classical projection-based model reduction approaches are designed to be minimal-error reduced-order models (ROMs). These ROMs often violate the underlying geometric structure which leads to unphysical numerical predictions, see \cite{peng2016symplectic}. Therefore, when the original system possesses specific qualitative features, it is preferable to construct a ROM that retains those features. The symplectic model reduction of Hamiltonian systems was introduced in \cite{peng2016symplectic}, where the Galerkin projection-based ROM was modified so that the ROM retains the underlying symplectic structure. Building on this work, the symplectic model reduction approach was combined with nonorthonormal bases in \cite{buchfink2019symplectic}. A similar structure-preserving approach with shifted snapshots was presented in \cite{gong2017structure} to improve the Hamiltonian approximation. The work in \cite{afkham2017structure} presented a reduced basis method approach for structure-preserving model reduction of parametric Hamiltonian systems. The reduced basis method has been extended to Hamiltonian systems with a more general Poisson structure in \cite{hesthaven2021structure}. A dynamical reduced basis method has been presented in \cite{pagliantini2021dynamical} for Hamiltonian systems with local low-rank structure. The idea of structure-preservation is explored at the variational formulation level in \cite{egger2021energy} to deduce important properties about POD-based model reduction of Hamiltonian systems. The above methods laid the foundation for structure-preserving model reduction for Hamiltonian systems, but they do require full access to the computer model, which is often not possible or feasible when working with proprietary, or very complex computer code. Data-driven (a.k.a nonintrusive) reduced modeling methods do not require such access, and are therefore an attractive alternative. 
\par 
For Hamiltonian systems, a variety of structure-preserving data-driven approaches have been developed recently, e.g., Hamiltonian neural networks \cite{greydanus2019hamiltonian}, symplectic networks \cite{jin2020sympnets}, Gaussian processes \cite{bertalan2019learning}, Bayesian system identification \cite{galioto2020bayesian}, and orthogonal polynomials \cite{wu2020structure}. The majority of these approaches are only concerned with learning Hamiltonian systems when the data is coming from very low-dimensional systems, i.e. 3-4 dimensions. This inability to learn from high-dimensional data limits their use for learning models from data of large-scale systems such as semi-discretized PDEs. On the other hand, the combination of data reduction and model reduction---termed data-driven reduced-order modeling---is a feasible approach for this setting. 
\par
For linear systems, a variety of successful data-driven model reduction approaches have been developed, e.g. the Loewner framework \cite{ionita2014data}, eigensystem realization \cite{kung1978new,kramer2016tangential,ma2011reduced}, vector fitting \cite{drmac2015quadrature}, but methods for learning ROMs for nonlinear systems in a nonintrusive way is still a burgeoning research area. For nonlinear systems, nonintrusive model reduction generally involves choosing a particular parametrization of the nonlinear terms. The Loewner approach has been extended to bilinear and quadratic-bilinear systems in \cite{antoulas2016model,gosea2018data}. Dynamic mode decomposition (DMD) has also been used for learning linear ROMs for nonlinear systems in \cite{rowley2009spectral,schmid2010dynamic}.  It is worth mentioning that sparsity-promoting regression techniques have been used in \cite{brunton2016discovering,rudy2017data, schaeffer2018extracting} for data-driven discovery of governing equations from a dictionary of nonlinear candidate functions. However, these sparse approximation approaches are not used for reducing the dimension of large-scale systems. 
\par 
Operator inference for nonintrusive model reduction was introduced in \cite{peherstorfer2016data} and applied to full-order models (FOMs) that are linear or have low-order polynomial nonlinear terms. Using lifting transformations, the operator inference framework has been extended to general nonlinear systems in \cite{SKHW2020_learning_ROMs_combustor,qian2019transform,QKPW2020_lift_and_learn}. The approach has also been extended to a gray-box setting in \cite{benner2020operator} where analytical expressions for the nonpolynomial nonlinear terms are known and the remaining operators are learned via operator inference. Convergence and accuracy certificates were developed in \cite{peherstorfer2020sampling,uy2021operator}.
\par
Our goal is to efficiently and stably learn Hamiltonian reduced-order models from high-dimensional data.  We approach this problem by proposing the nonintrusive \textit{Hamiltonian operator inference} (H-OpInf), a structure-preserving data-driven model reduction method that preserves the underlying symplectic structure inherent to Hamiltonian systems. The method can work with high-dimensional state-trajectory data from a Hamiltonian system. We project this data onto a low-dimensional basis via symplectic projection, and learn the reduced Hamiltonian operators from the reduced data using a constrained least-squares operator inference procedure that ensures that the models \textit{preserve} the Hamiltonian nature of the problem. 
\par 
The remainder of the paper is organized as follows. Section~\ref{s:2} reviews the basics of Hamiltonian PDEs and describes intrusive structure-preserving model reduction. Section~\ref{s:3}  presents the proposed structure-preserving operator inference problem for Hamiltonian systems with nonpolynomial nonlinearities. In Section~\ref{s:4} we apply our proposed method to three Hamiltonian systems with increasing complexity: the linear wave equation, the cubic Schr\"{o}dinger equation and the sine-Gordon equation. Our numerical results demonstrate the learned models' interpretability and ability to provide accurate long-time prediction beyond the training data. Finally, in Section~\ref{s:5} we provide concluding remarks and future research directions.
%
\section{Background}
\label{s:2}
In this section, we introduce Hamitonian PDE models and describe intrusive projection-based model reduction for Hamiltonian systems. This provides the necessary background for our nonintrusive method in Section~\ref{s:33}. In Section \ref{s:21} we first review the basics of Hamiltonian PDEs by deriving the governing PDEs, followed by their structure-preserving space discretization. After deriving the FOM equations, we closely follow \cite{peng2016symplectic} to derive projection-based Hamiltonian ROMs via symplectic projection in Section \ref{s:22}. 
%
\subsection{Hamiltonian Systems}
\label{s:21} 
We consider a general infinite-dimensional Hamiltonian system described by the following evolutionary PDE 
\begin{equation}
\frac{\partial  y(x,t)}{\partial t} = \mathcal{S} \frac{\delta \mathcal{H}}{\delta y},
\label{eq:HPDE}
\end{equation}
where $x$ is the spatial variable, $t$ is time, $\mathcal{S}$ is a skew-symmetric operator,  $\frac{\delta \mathcal{H}}{\delta y}$ is the variational derivative\footnote{The variational derivative of $\mathcal{H}$ is defined through $\frac{d}{d\epsilon}\mathcal{H}[y+ \epsilon v]|_{\epsilon=0}=\bigg\langle \frac{\delta \mathcal{H}}{\delta y},v\bigg\rangle$ where $v$ is an arbitrary function.} of $\mathcal{H}$, and we consider Hamiltonian functional $\mathcal{H}$ defined by 
\begin{equation}
\cH[y]=\int \left(H_{\text{quad}}(y, y_x,\cdots) + H_{\text{nl}}(y) \right) \ \text{d}x,
\label{eq:Hc}
\end{equation}
where $y_x=\frac{\partial y}{\partial x}$ is the partial derivative of $y$ with respect to $x$, $H_{\text{quad}}(y,y_x,\cdots)$ contains quadratic terms and $H_{\text{nl}}(y)$ contains spatially local nonlinear terms of the Hamiltonian functional. 
\begin{remark}
Although the partition of the integrand in \eqref{eq:Hc} (quadratic terms in $H_{\text{quad}}$ as a function of the spatial derivatives, and spatially local nonlinear terms in $H_{\text{nl}}$ as a function of state variables) seems very restrictive, the Hamiltonian functional form covers most (if not all) of the Hamiltonian PDEs found in science and engineering applications, see \cite{bridges2006numerical,celledoni2012preserving}. Importantly, a quadratic component exists in almost all Hamiltonian systems.
\end{remark}

Hamiltonian PDEs possess important geometric properties and their numerical simulation consists of two steps: 1) structure-preserving space discretization that reduces the Hamiltonian PDE to a system of Hamiltonian ODEs; 2) structure-preserving time integration of the finite-dimensional Hamiltonian ODE. We will briefly discuss those in the following.
%
%
\subsubsection{Space Discretization of Hamiltonian PDEs}
\label{s:211}
Space-discretized Hamiltonian FOMs are usually derived from the PDE by finite difference or pseudo-spectral methods. The most popular approach to obtain a Hamiltonian FOM from the infinite-dimensional Hamiltonian system \eqref{eq:HPDE} is to discretize the space-time continuous Hamiltonian functional \eqref{eq:Hc} directly. The key steps in this approach are discussed in \cite{bridges2006numerical,celledoni2012preserving}. The resulting finite-dimensional Hamiltonian model can be described by 
\begin{equation}
\dot{\y}=\mathbf{S}_{\rm d} \nabla_{\y}\dH(\y),
\end{equation}
where $\mathbf{S}_{\rm d}=-\mathbf{S}_{\rm d}^\top$ and $\dH$ is the space-discretized Hamiltonian function. 
\par
For this work, we will focus on canonical Hamiltonian systems, i.e. $\mathbf S_{\rm d}=\Jn=\begin{bmatrix}
       \bzero & \In \\
       -\In & \bzero
     \end{bmatrix}$ where $\In$ is the $n \times n$ identity matrix. For canonical systems, the state vector $\y \in \Real^{2n}$ can be partitioned as $\y=[\q^\top,\p^\top]^\top$ where $\q,\p \in \Real^n$. Both $\q$ and $\p$ have distinct physical interpretations, and their relation to each other induces the canonical symplectic structure. The governing equations for the semi-discrete canonical Hamiltonian systems are 
      \begin{equation}
    \dot{\y}=\begin{bmatrix}
       \dot{\q} \\
       \dot{\p}
     \end{bmatrix}=\Jn \nabla_{ \y }\dH(\q,\p)
    = \begin{bmatrix}
       \nabla_{ \p }\dH(\q,\p) \\
       - \nabla_{ \q }\dH(\q,\p)
     \end{bmatrix}.
     \label{eq:HFOM}
\end{equation}
In addition to retaining the Hamiltonian character of the given PDE, the structure-preserving space discretizations often introduce additional mathematical structure in the FOM operators, see \cite{li2017spectral}. To illustrate this, we present a simple example of a Hamiltonian PDE.

\begin{example}
\label{ex}
Consider the one-dimensional nonlinear wave equation with wave speed $c$ which has the Hamiltonian
\begin{equation}
\cH[q,p]=\int \left( \frac{1}{2}p^2 + \frac{c^2}{2}q_x^2   + H_{\text{nl}}(q,p) \right).
\end{equation}
Direct discretization of the Hamiltonian functional with $n$ equally spaced grid points leads to the following space-discretized Hamiltonian 
\begin{equation}
\dH(\q,\p)=\sum_{i=1}^n \left[\frac{1}{2}p_i^2 + \frac{c^2}{2}\left(\sum_{j=1}^nM_{ij}q_j \right)^2 + H_{\text{nl}}(q_i,p_i) \right],
 \end{equation}
 where $q_i:=q(t,x_i)$, $p_i:=p(t,x_i)$, and the derivative of $q$ with respect to $x$ is approximated by an appropriate differentiation matrix $\mathbf{M}=\left(M_{ij} \right)_{i,j=1}^n$, i.e., $ q_x(x_i) \approx \sum_{j=1}^nM_{ij}q_j$. For $\dx \to 0$ with $n\dx=\ell$, the term $\dH(\q,\p)\dx$ converges to the space-time continuous Hamiltonian functional $\cH$. The governing FOM equations for the Hamiltonian system are 
\begin{equation}
\begin{bmatrix}
       \dot{\q} \\
       \dot{\p}
     \end{bmatrix}
= \begin{bmatrix}
      \bzero & \In \\
      c^2\mathbf{D} & \bzero
     \end{bmatrix} \begin{bmatrix}
       \q \\
       \p
     \end{bmatrix} + \begin{bmatrix}
       \nabla_{ \p }H_{\text{nl}}(\q,\p) \\
       - \nabla_{ \q }H_{\text{nl}}(\q,\p)
     \end{bmatrix}, 
\end{equation}
and regardless of the spatial derivative approximation, the linear FOM operators are always symmetric, i.e., $\mathbf{D}=\mathbf{M}^\top\mathbf{M}$.
\end{example}

\subsubsection{Time Integration of Semi-discrete Hamiltonian Systems}
Once a Hamiltonian FOM has been formulated, a structure-preserving method can be applied in time to complete the structure-preserving discretization in space and time. The flow map for Hamiltonian FOMs \eqref{eq:HFOM} preserves the canonical symplectic form and conserves the system Hamiltonian $\dH$, i.e., $\dH(\q(0),\p(0))=\dH(\q(t),\p(t))$ for all $t$. The field of geometric numerical integration methods has shown that it is advantageous to use time integrators that preserve these two geometric features. In fact, time integrators that do not respect the underlying geometric structure lead to unphysical numerical results. The work in \cite{zhong1988lie} showed that a numerical integrator with a fixed time step cannot preserve the symplectic form and conserve the energy simultaneously for general Hamiltonian systems. Based on this result, structure-preserving time integrators for canonical Hamiltonian systems can be divided into two categories: (i) energy-preserving integrators and (ii) symplectic integrators. Both approaches have their own advantages and the preferred geometric numerical integration method depends on the Hamiltonian system. Energy-preserving integrators guarantee that the numerical solution is restricted to a codimension 1 submanifold of the configuration manifold whereas symplectic integrators ensure a more global and multi-dimensional behavior through symplectic structure preservation.  
%
%
\subsection{Intrusive Structure-preserving Model Reduction}
\label{s:22} 
In projection-based model reduction, the semi-discrete model is projected onto a low-dimensional subspace. The key idea in structure-preserving model reduction is to preserve the underlying geometric structure during the projection. Since the FOM is a Hamiltonian system with underlying symplectic structure, the projection step is treated as the symplectic inverse of a symplectic lift from the low-dimensional subspace to the state space, see \cite{peng2016symplectic}. A \textit{symplectic lift} is defined by $\y=\V\tilde{\y}$ where $\V \in \Real^{2n \times 2r}$ is a \textit{symplectic matrix}, i.e., a matrix that satisfies 
\begin{equation}
\V^\top\Jn\V=\Jr.
\end{equation}
The \textit{symplectic inverse} $\Vt$ of a symplectic matrix $\V$ is defined by 
\begin{equation}
\Vt=\Jr^\top\V^\top\Jn,
\end{equation}
and the \textit{symplectic projection} can be written as $\tilde{\y}=\Vt\y$. The time evolution of the reduced state $\dot{\tilde{\y}}$ is given by
\begin{equation}
\dot{\tilde{\y}}=\Vt\dot{\y}=\Vt\Jn \nabla_{ \y }\dH(\q,\p)=\Jr\V^\top\nabla_{\y}\dH(\y)=\Jr\nabla_{\tilde{\y}}\dH(\V\tilde{\y})
\end{equation}
where we have used the chain rule $\nabla_{\tilde{\y}}\dH(\V\tilde{\y})=\V^\top\nabla_{\y}\dH(\y)$ in the last step. The symplectic Galerkin projection of a $2n$-dimensional Hamiltonian system \eqref{eq:HFOM} is given by a $2r$-dimensional $(r \leq n)$ system 
\begin{equation}
\dot{\tilde{\y}}=\Jr\nabla_{\tilde{\y}}\tilde{H}(\tilde{\y}),
\end{equation}
where $\tilde{\y}$ is the reduced state vector with the reduced Hamiltonian $\tilde H(\tilde{\y}):=\dH (\V\tilde{\y})$. While the symplectic Galerkin projection approach yields reduced systems that retain the Hamiltonian nature, the reduced Hamiltonian $\tilde H$, through its definition in terms of FOM Hamiltonian $\dH$, requires access to FOM operators. 
\par
Proper symplectic decomposition (PSD) is a method to find a symplectic projection matrix $\V$ that simultaneously minimizes the projection error in a least-squares sense, i.e., 
\begin{equation}
\min_{\substack{\V \\ \text{s.t.} \ \V^\top\Jn\V=\Jr}} || \mathbf{Y} - \V\Vt \mathbf{Y} ||_F.
\label{eq:PSD}
\end{equation}
where $\mathbf{Y}:=[\y(t_1), \cdots, \y(t_K)] \in \Real^{2n \times K}$ is the snapshot data matrix, and $||\cdot ||_F$ is the Frobenius norm. Since solving \eqref{eq:PSD} to obtain the symplectic basis matrix $\V$ is computationally expensive, we briefly  outline three efficient algorithms, first presented in \cite{peng2016symplectic}, for finding approximated optimal solution for the symplectic matrix $\V$. These algorithms search for a near-optimal solution over different subsets of $Sp(2r,\Real^{2n})$ where $Sp(2r,\Real^{2n})$ is the set of all $2n \times 2r$ symplectic matrices. 
\begin{enumerate}
\item \textbf{Cotangent lift:} This algorithm computes SVD of the extended snapshot matrix \\$\mathbf{Y}_1:=[\q(t_1), \cdots,\q(t_K), \p(t_1), \cdots,\p(t_K)]$  to obtain a POD basis matrix $\mathbf{\Phi} \in \Real^{n \times r}$ and then constructs the symplectic basis matrix $\V_1= \begin{bmatrix}
      \mathbf{\Phi} & \bzero \\
     \bzero & \mathbf{\Phi}
     \end{bmatrix}$ with $\Vt_1=\V_1^\top$. The diagonal nature of $\V_1$ ensures that the interpretability of $\q$ and $\p$ is retained in the reduced setting. 
\item \textbf{Complex SVD:} This algorithm describes the solution in the phase space by $\q(t_i) + i\p(t_i)$ to build a complex snapshot matrix $\mathbf{Y}_2:=[\q(t_1) + i\p(t_1), \cdots,\q(t_K)+ i\p(t_K)]$ and then computes the complex SVD of $\mathbf{Y}_2$ to obtain a basis matrix $\mathbf{\Phi} + i\mathbf{\Psi} \in \mathbb{C}^{n \times r}$.  The symplectic basis matrix is $\V_2= \begin{bmatrix}
      \mathbf{\Phi} & -\mathbf{\Psi} \\
     \mathbf{\Psi} & \mathbf{\Phi}
     \end{bmatrix}$ with $\Vt_2=\V_2^\top$. Due to the nonzero $\mathbf{\Psi}$ matrices on the off-diagonals in $\V_2$, the complex SVD algorithm loses the distinction between $\q$ and $\p$ in the symplectic projection step.
\item \textbf{Nonlinear programming:} This approach starts with an intermediate symplectic matrix $\V_{\text{int}} \in \Real^{2n \times 2k}$ with $k>r$ obtained by cotangent lift or complex SVD and assumes that the optimal solution $\V_3$ is a linear transformation of $\V_{\text{int}}$. This assumption simplifies the original nonlinear programming problem to a smaller nonlinear programming problem
\begin{equation}
\min_{\substack{\mathbf{C} \\ s.t. \ \mathbf{C}^\top\mathbf{J}_{2k}\mathbf{C}=\Jr}} || \mathbf{Y} - \V_{\text{int}}\mathbf{C}\mathbf{C}^+\V_{\text{int}}^+\mathbf{Y} ||_F,
\end{equation}
to obtain  $\mathbf{C}$. The symplectic basis matrix is $\V_3=\V_{\text{int}} \cdot \mathbf{C}$. In addition to being computationally expensive, the nonlinear programming approach also loses the physical meaning of the states $\q$ and $\p$ in the reduced setting.
\end{enumerate}
Since all three PSD approaches restrict to a specific subset of symplectic basis matrices, $Sp(2r,\Real^{2n})$, the resulting basis matrix might be globally suboptimal. A new technique based on an SVD-like decomposition to derive a non-orthogonal symplectic basis was presented in \cite{buchfink2019symplectic}. However, the optimality with respect to the PSD projection error \eqref{eq:PSD} of this non-orthonormal symplectic basis is still an open question. More recently, \cite{gao2021riemannian} developed Riemannian optimization methods for optimization problems with symplectic constraints, and these methods have been proven to converge globally to critical points of the objective function.
%
%
\section{Hamiltonian Operator Inference}
\label{s:3}
In this section, we propose H-OpInf, a Hamiltonian operator inference framework for canonical Hamiltonian PDEs with general nonpolynomial nonlinearities as in \eqref{eq:Hc}. In Section \ref{s:31} we first motivate the need for H-OpInf by demonstrating how the standard operator inference from \cite{peherstorfer2016data} fails to preserve the underlying symplectic structure. Based on the observations from this motivating example, we present H-OpInf in Section \ref{s:32}, a nonintrusive physics-preserving method to learn ROMs of Hamiltonian systems. We then show in Section \ref{s:33} that under certain conditions, the learned operators from nonintrusive H-OpInf converge to their intrusive projection-based counterparts. We also discuss the overall computational procedure of H-OpInf in Section \ref{s:34}.
%
\subsection{Motivation}
\label{s:31} 
We revisit the wave equation example from  Example \ref{ex} with $H_{\text{nl}}=0$, i.e., the linear wave equation 
\begin{equation}
    \frac{\partial^2 \varphi}{\partial t^2}=c^2 \frac{\partial^2 \varphi}{\partial x^2}, 
\end{equation}
defined on $x \in [0,\ell]$. With $q=\varphi$ and $p=\varphi_t$, the associated Hamiltonian functional is given by
\begin{equation}
    \cH(q,p)=\int^{\ell}_0 \left[ \frac{1}{2}p^2 + \frac{1}{2}c^2q_x^2  \right] \ \text{d}x,
\end{equation}
and the original PDE can be recast as a Hamiltonian PDE
\begin{equation}
    \dot{q}=\frac{\delta \cH}{\delta p}, \quad \quad  \dot{p}=-\frac{\delta \cH}{\delta q}.
\end{equation}
With $n$ equally spaced grid points, we use finite difference for spatial discretization to obtain the space-discretized FOM Hamiltonian 
\begin{equation}
    \dH( \textbf{y} )=\sum^n_1 \left[ \frac{1}{2}p_i^2 + \frac{c^2 (q_{i+1}-q_i)^2}{4\dx^2} + \frac{c^2 (q_{i}-q_{i-1})^2}{4\dx^2}  \right] ,
\end{equation}
where $q_i:=\varphi(t,x_i)$, $p_i:=\varphi_t(t,x_i)$, and $\y=[q_1, \ldots, q_n, p_1, \ldots, p_n]^\top$. For $\dx \to 0$ with $n\dx=\ell$, $\dH\dx$ converges to the space-time continuous Hamiltonian functional $\cH$. The FOM equations are given by the Hamilton's equations for $\dH$,
\begin{equation}
    \dot{\y}=\Jn \nabla_{ \y }\dH(\y)=\begin{bmatrix}
       \bzero & \In \\
       -\In & \bzero
     \end{bmatrix} \begin{bmatrix}
       -c^2\mathbf{D}_{\text{fd}} & \bzero \\
       \bzero & \In
     \end{bmatrix}\y =\begin{bmatrix}
       \bzero & \In \\
       c^2\mathbf{D}_{\text{fd}} & \bzero
     \end{bmatrix} \y, 
     \label{eq:LW_FOM}
\end{equation}
where $\mathbf D_{\text{fd}}$ denotes the finite difference approximation for the spatial derivative $\partial_{xx}$. 
\par
We now set $\ell=1$ and the wave speed to $c=0.1$, which leads to the same linear FOM with periodic boundary condition as in \cite{peng2016symplectic}. To define the initial conditions, we need the following ingredients. Let $s(x)=10 |x-\frac{1}{2}|$, and consider the following cubic spline function over $s$:
\begin{equation*}
h(s)= \begin{cases} 
      1-\frac{3}{2}s^2  + \frac{3}{4}s^3& 0\leq s\leq 1, \\
      \frac{1}{4}(2-s)^3 & 1\leq s\leq 2, \\
      0 & s>2.
   \end{cases}
\end{equation*} 
The initial conditions are $\q(0)=[h(s(x_1)),\cdots,h(s(x_n))]$, and $\p(0)=\mathbf{0}$. We choose $n=500$ grid points leading to a discretized state $\y \in \Real^{1000}$. The FOM is numerically integrated until time $T=10$ using the implicit midpoint rule, for which the time-marching equations are
\begin{equation}
\frac{\y_{k+1}-\y_k}{\dt}  
= \Jn\nabla_{\y}\dH\left( \frac{\y_{k+1} + \y_k}{2} \right),
\label{eq:symp}
\end{equation}
with fixed time step $\dt$. The resulting time integrator is a second-order scheme and can also be used for dynamical systems that are not Hamiltonian. For this example, we choose $\dt=0.01$. To propagate the system forward in time, we need to solve a system of $2n=1000$ linear equations at every time step.
\par
We apply standard operator inference \cite{peherstorfer2016data} to the linear wave equation to demonstrate how violating the underlying symplectic structure leads to unstable ROMs. Based on the snapshot data $\Y$ from the FOM numerical simulation, we compute the POD basis $\V$ from $\Y$ via SVD where the POD basis vectors are columns of the POD basis matrix
 \begin{equation*}
 \V=[\v_1, \cdots, \v_{2r}] \in \Real^{2n \times 2r}.
 \end{equation*}
We also obtain the time-derivative data $\dot{\Y}$ from the snapshot data $\Y$ using a finite difference approximation. We then obtain the reduced state trajectory data $\hat{\Y}$ and the reduced time-derivative data $\dot{\hat{\Y}}$ via projections onto the POD basis matrix $\V$
 \begin{equation*}
 \hat{\Y}=\V^\top\Y, \quad \quad  \dot{\hat{\Y}}=\V^\top\dot{\Y}.
 \end{equation*} 
Based on the linear wave FOM equation, we postulate a linear model form $\dot{\hat{\y}}=\hat{\mathbf D}\hat{\y}$ for learning the ROM. Thus, for the linear wave equation, the standard operator inference problem essentially involves solving the following least-squares problem for $\hat{\mathbf D}$
\begin{equation}
\min_{\hat{\mathbf D}} || \dot{\hat{\Y}} - \hat{\mathbf D} \hat{\Y} ||_F.
\end{equation}
\par
Figure \ref{fig: stacked_state} shows state approximation error for ROMs of different sizes. We see that the state error decreases up to $2r=20$ and oscillates from $2r=20$ to $2r=40$. Although the state error results over the training data indicate decrease in the state error with increasing reduced dimension, the ROMs do not conserve the space-discretized Hamiltonian $\dH$. The increasing FOM energy error in Figure \ref{fig: stacked_Hd} confirm the fact that the reduced solution trajectories do not conserve the system energy. In fact, the energy error plot for long-time simulation outside the training data shows that the learned ROMs lead to unphysical predictions. In addition to violating the underlying Hamiltonian structure, the learned operator $\hat{\mathbf D}$ contains nonzero matrices on the diagonals (not shown here), which also illustrates that the standard operator inference does not preserve the block structure that the system should have, see \eqref{eq:LW_FOM}. The coupling structure in Hamiltonian systems is intrinsically connected to the physical variables $\q$ and $\p$ in the state vector. Standard operator inference loses the physical meaning of the states $\q$ and $\p$ in the model reduction process. 
\begin{figure}
\captionsetup[subfigure]{oneside,margin={1.5cm,0 cm}}
\begin{subfigure}{.38\textwidth}
       \setlength\fheight{4 cm}
        \setlength\fwidth{\textwidth}
%
%
\begin{tikzpicture}

\begin{axis}[%
width=0.974\fwidth,
height=\fheight,
at={(0\fwidth,0\fheight)},
scale only axis,
xmin=0,
xmax=40,
xlabel style={font=\color{white!15!black}},
xlabel={\small Reduced Dimension $2r$},
ymode=log,
ymin=0.01,
ymax=1.06384300836693,
yminorticks=true,
ylabel style={font=\color{white!15!black}},
ylabel={\small Relative State Error},
axis background/.style={fill=white},
xmajorgrids,
ymajorgrids
]
\addplot [color=blue, dashed,line width=1.5pt, mark=o, mark options={solid, blue}, forget plot]
  table[row sep=crcr]{%
4	1.06384300836694\\
8	0.430572590951954\\
12	0.107690160943356\\
16	0.0363756372170366\\
20	0.0151019448993194\\
24	0.0234466439183185\\
28	0.0887799015988189\\
32	0.018308867290126\\
36	0.051127987578703\\
40	0.0723924268283788\\
};
\end{axis}

\begin{axis}[%
width=1.257\fwidth,
height=1.257\fheight,
at={(-0.163\fwidth,-0.163\fheight)},
scale only axis,
xmin=0,
xmax=1,
ymin=0,
ymax=1,
axis line style={draw=none},
ticks=none,
axis x line*=bottom,
axis y line*=left
]
\end{axis}
\end{tikzpicture}%
\caption{State error (training data)}
\label{fig: stacked_state}
    \end{subfigure}
    \hspace{1.4cm}
    \begin{subfigure}{.38\textwidth}
           \setlength\fheight{4 cm}
           \setlength\fwidth{\textwidth}
\input{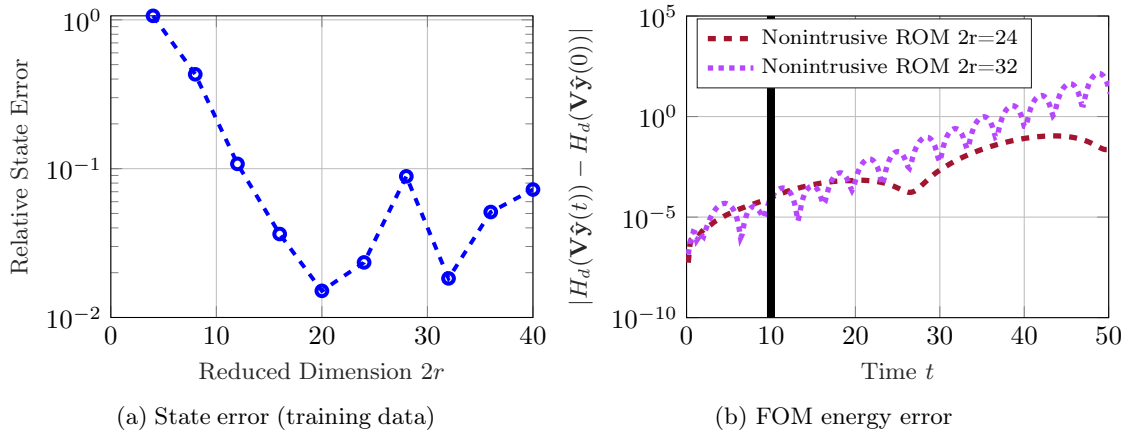}
\caption{FOM energy error}
\label{fig: stacked_Hd}
    \end{subfigure}
    \caption{Linear wave equation: Even though plot (a) shows low approximation error in the training data regime, the corresponding FOM energy error behavior in plot (b) reveals that the standard operator inference violates the underlying Hamiltonian structure. The black line indicates end of training time interval. See Section \ref{s:41} for numerical implementation details. }
\label{fig:motivation}
\end{figure}
%
%
\subsection{Hamiltonian Operator Inference}
\label{s:32} 
In this work, we consider the situation that we have a Hamiltonian PDE model \eqref{eq:HPDE} with canonical structure, and  we have simulated data thereof. The goal of this work is to learn a Hamiltonian ROM from data of a canonical Hamiltonian system, so that the learned ROM: 
\begin{enumerate}
\item is a canonical Hamiltonian system;
\item retains the physical interpretation of the state variables and preserves the coupling structure;
\item respects the symmetric property of structure-preserving space discretizations.
\end{enumerate}
In addition to the canonical Hamiltonian structure, we also assume that we have knowledge about $H_{\text{nl}}$ at the PDE level ,  which is in line with our gray-box setting \eqref{eq:Hc}. If we consider the general wave equation from Section \ref{s:211}, then the nonlinear component $H_{\text{nl}}(q,p)$ of the space-time continuous Hamiltonian functional \eqref{eq:Hc} is assumed to be given explicitly, whereas the  quadratic terms in $H_{\text{quad}}$ and details about their spatial discretization are unavailable. 
\par
Next, we introduce our new framework. Let $\y_1, \cdots, \y_k$ be the solutions of the Hamiltonian FOM at $t_1,\cdots, t_k$ computed with a structure-preserving numerical integration scheme and initial condition $\y_0$. We define the snapshot matrices
\begin{equation}
\Q=\begin{bmatrix}
       \q_1  \cdots \q_K 
     \end{bmatrix} \in \Real^{n \times K}, \quad \quad \Pp=\begin{bmatrix}
      \p_1  \cdots \p_K 
     \end{bmatrix} \in \Real^{n \times K}.
     \label{eq:snapshot}
\end{equation}
Assuming knowledge about $H_{\text{nl}}$ in \eqref{eq:Hc}, we define the nonlinear forcing $\fq(\y)$ and $\fp(\y)$ as 
\begin{equation}
\fq(\y)=\begin{bmatrix}
       \frac{\partial H_{\text{nl}}}{\partial p_1} (q_1,p_1)  \cdots \frac{\partial H_{\text{nl}}}{\partial p_n} (q_n,p_n)
     \end{bmatrix}^\top \in \Real^{n}, \quad \quad \fp(\y)=\begin{bmatrix}
       \frac{\partial H_{\text{nl}}}{\partial q_1} (q_1,p_1)  \cdots \frac{\partial H_{\text{nl}}}{\partial q_n} (q_n,p_n)
     \end{bmatrix}^\top \in \Real^{n }.
     \label{eq:fq}
\end{equation}
We utilize the explicit form of $\fp$ and $\fq$ to define the forcing snapshot matrices
\begin{equation}
\Fq=\begin{bmatrix}
       \fq(\y_1)  \cdots   \fq(\y_K)
     \end{bmatrix} \in \Real^{n \times K}, \quad \quad \Fp=\begin{bmatrix}
    \fp(\y_1)  \cdots   \fp(\y_K)
     \end{bmatrix} \in \Real^{n \times K}.
     \label{eq:F}
\end{equation}
We also compute the time-derivative data $\dot{\q}$ and $\dot{\p}$ from the state trajectory data $\q$ and $\p$ using a finite difference scheme to build the snapshot matrices of the time-derivative data
\begin{equation}
\dot{\Q}=\begin{bmatrix}
       \dot{\q}_1  \cdots \dot{\q}_K 
     \end{bmatrix} \in \Real^{n \times K}, \quad \quad \dot{\Pp}=\begin{bmatrix}
        \dot{\p}_1  \cdots \dot{\p}_K 
     \end{bmatrix} \in \Real^{n \times K}.
     \label{eq:dQ}
\end{equation}
Given these snapshot matrices, our goal is to learn a Hamiltonian ROM directly from the data. To learn the reduced operators, we propose to project FOM trajectories onto low-dimensional symplectic subspaces of the high-dimensional state spaces and then fit operators to the projected trajectories in a structure-preserving way. For the symplectic projection step, we choose the cotangent lift algorithm to generate our symplectic basis matrix. In addition to being a symplectic basis matrix, the specific block structure of the basis matrix allows us to retain physical interpretation of $\q$ and $\p$ variables in the reduced setting, i.e., 
\begin{equation}
\begin{bmatrix}
      \q \\
     \p
 \end{bmatrix}\approx \begin{bmatrix}
      \mathbf{\Phi} & \bzero \\
     \bzero & \mathbf{\Phi}
 \end{bmatrix}\begin{bmatrix}
      \qhat \\
     \phat
 \end{bmatrix}.
\end{equation}
We obtain projections of the trajectory snapshot data via the projections onto the symplectic basis matrix $\Vq=\Vp=\mathbf{\Phi} \in \Real^{n \times r}$ obtained via the cotangent lift algorithm,
\begin{equation}
\Qhat= \Vq^\top\Q \in \Real^{r \times K}, \quad \quad \Phat=\Vp^\top\Pp \in \Real^{r \times K}.
\label{eq:Yhat}
\end{equation}
Similarly, we obtain projections of the nonlinear forcing data to obtain 
\begin{equation}
\Fhatq= \Vp^\top\Fq \in \Real^{r \times K}, \quad \quad \Fhatp=\Vq^\top\Fp \in \Real^{r \times K}.
\label{eq:Fhat}
\end{equation}
We also compute projections of the time-derivative data to obtain the reduced time-derivative data
\begin{equation}
\dot{\Qhat}= \Vq^\top\dot{\Q} \in \Real^{r \times K}, \quad \quad \dot{\Phat}=\Vp^\top\dot{\Pp} \in \Real^{r \times K}.
     \label{eq:dQhat}
\end{equation}
Inspired by the knowledge of the Hamiltonian functional, we define the following reduced Hamiltonian in terms of the inferred reduced operators $\Dhatq \in \Real^{r \times r}$ and $\Dhatp\in \Real^{r \times r}$
\begin{equation} 
 \Hr(\qhat,\phat)= \frac{1}{2}\qhat^\top \Dhatq \qhat
 + \frac{1}{2}\phat^\top \Dhatp \phat
 + \Hr_{\text{nl}}(\qhat,\phat).
 \label{eq:Hrmodel}
\end{equation}
Based on the assumed form for $\Hr(\qhat,\phat)$, we derive the ROM equations of motion
\begin{gather*}
\dot{\qhat} =\frac{\partial  \Hr }{\partial \phat}=\Dhatp \phat +  \Vp^\top\fq(\Vq \qhat,\Vp \phat),     \\
\dot{\phat} =- \frac{\partial  \Hr }{\partial  \qhat}=-\Dhatq \qhat-  \Vq^\top\fp(\Vq \qhat,\Vp \phat).
\end{gather*}
Using this Hamiltonian ROM form, we propose to solve the following optimization problem to compute $\Dhatq$ and $\Dhatp$ 
\begin{equation}
\label{eq:H-OpInf_opt}
\min_{\substack{\Dhatq=\Dhatq^\top,\\ \Dhatp=\Dhatp^\top}} \bigg| \bigg| \begin{bmatrix}
\dot{\Qhat}- \Fhatq(\Qhat,\Phat) \\ \dot{\Phat} + \Fhatp(\Qhat,\Phat) 
\end{bmatrix} -  
\begin{bmatrix} 
\bzero & \Dhatp \\ 
-\Dhatq & \bzero
\end{bmatrix}
\begin{bmatrix}
 \Qhat\\ \Phat
\end{bmatrix} 
\bigg| \bigg| _F.
\end{equation}
The symmetric constraints on $\Dhatq$ and $\Dhatp$ ensure that the learned reduced operators retain the symmetric property of the full-model operators introduced during the structure-preserving spatial discretization, see Section \ref{s:211}. The symmetric reduced operators $\Dhatq$ and $\Dhatp$ learned via H-OpInf yield structure-preserving ROMs that are Hamiltonian systems.
\begin{remark} 
    We have introduced the H-OpInf framework in the context of canonical Hamiltonian systems. Some Hamiltonian systems, such as the Korteveg-de-Vries (KdV) equation, Burgers' equation, and Maxwell's equations,   possess a more general Hamiltonian structure, i.e., $\mathbf{S}_{\rm d}=-\mathbf{S}_{\rm d}^\top \neq \Jn$. Using an even number of spatial grid points, the noncanonical Hamiltonian PDE yields an even-dimensional FOM which can be then transformed to the canonical form with $\mathbf{S}_{\rm d}=\Jn$ using a congruent transformation. This transformation requires access to FOM operators. However, if the FOM is given in the transformed canonical form then the H-OpInf framework directly carries over to such transformed systems. Alternatively, if we assume access to such a congruent transformation then our H-OpInf method may be combined with the lifting transformations as in \cite{SKHW2020_learning_ROMs_combustor,qian2019transform, QKPW2020_lift_and_learn}  to derive Hamiltonian ROMs for noncanonical Hamiltonian systems. 
\end{remark}
%
\subsection{Theoretical Result}
\label{s:33} 
We show that under certain conditions on the time discretization, the nonintrusive Hamiltonian ROM operators via H-OpInf converge to the  intrusive Hamiltonian ROM operators. In order to obtain these results, we make the following two assumptions on the time discretization of the FOM and ROM. 
\begin{assumption}
\label{as:1}
 The time stepping scheme for the FOM is convergent, i.e., 
\begin{equation}
\max_{i \in \{1, \cdots, T/\Delta t \}} \bigg|\bigg| \y_i - \y(t_i) \bigg|\bigg|_2 \to 0 \quad as \quad \dt \to 0,
\end{equation}
where $\y_i$ is the discrete state of the FOM system at time $t_i$ computed with a time stepping scheme. 
\end{assumption}
\begin{assumption}
\label{as:2}
 The derivatives approximated from projected states, $\dot{\hat{\y}}_k$, converge to $\frac{\text{d}}{\text{d}t}\hat{\y}(t_k )$ as the discretization time step $\dt \to 0$, i.e.,
\begin{equation}
\max_{i \in \{1, \cdots, T/\Delta t \}} \bigg|\bigg|\dot{\hat{\y}}_i - \frac{\text{d}}{\text{d}t}\hat{\y}(t_i) \bigg|\bigg|_2 \to 0 \quad as \quad \Delta t \to 0.
\end{equation}
\end{assumption}
\begin{theorem}
For a given symplectic basis matrix $\Vq=\Vp=\Phi \in \Real^{n \times r}$ obtained via the cotangent lift algorithm, let $\tilde {\mathbf D}_{\q}=\Vq^\top\mathbf D_{\q}\Vq$ and $\tilde {\mathbf D}_{\p}=\Vp^\top\mathbf D_{\p}\Vp$
be the intrusively projected ROM operators. If the data matrix has full column rank,
then for every $\epsilon > 0$, there exists $2r \leq 2n$ and a time step size $\dt> 0$ such that for the difference between the symmetric learned
operators $\Dhatq, \Dhatp$ and the symmetric (intrusive) projection-based $\tilde{\mathbf D}_{\q}, \tilde{\mathbf D}_{\p}$, we have
\begin{equation*}
|| \Dhatq- \tilde{\mathbf D}_{\q} ||_F \leq \epsilon, \quad \quad || \Dhatp- \tilde{\mathbf D}_{\p} ||_F \leq \epsilon.
\end{equation*}

\end{theorem}
\begin{proof} Consider a canonical Hamiltonian system with the following governing equations
\begin{equation} 
\dot{\q}=\mathbf D_{\p}\p + \fq(\q,\p),\quad \quad \dot{\p}=-\mathbf D_{\q}\q - \fp(\q,\p),
 \label{eq:Hdfom}
\end{equation}
where $\mathbf{D}_{\q},\mathbf{D}_{\p}$ are the linear symmetric FOM operators, and $\fq, \fp$ are the nonlinear forcing terms from \eqref{eq:fq}. Given FOM snapshot data $\Q = [\q_1, \cdots, \q_K] \in \Real^{n \times K}$ and $\Pp=[\p_1, \cdots, \p_K] \in \Real^{n \times K} $, we build $\dot{\Q}_{\rm ana}=[\dot{\q}_{1,\rm ana},\cdots, \dot{\q}_{K,\rm ana}]$ and $\dot{\Pp}_{\rm ana}=[\dot{\p}_{1,\rm ana},\cdots, \dot{\p}_{K,\rm ana}]$ by evaluating the Hamiltonian FOM vector field \eqref{eq:Hdfom}, i.e., 
\begin{equation*}
\dot{\q}_{k,\rm ana}=\mathbf D_{\p}\p_k + \fq(\q_k,\p_k),\quad \quad \dot{\p}_{k,\rm ana}=-\mathbf D_{\q}\q_k - \fp(\q_k,\p_k), \quad \quad k=1, \cdots, K.
\end{equation*} 
We also define a finite difference operator $I_{t}$ which operates on the FOM snapshot data to approximate the time-derivative data
\begin{equation*}
\dot{\Q}_{\rm ana}\approx I_t(\Q),\quad \quad  \dot{\Pp}_{\rm ana}\approx I_t(\Pp).
\end{equation*}
Using the definition of the finite difference operator $I_t$, the reduced time-derivative data from \eqref{eq:dQhat} can be written as $\dot{\hat \Q}=\Vq^\top I_t(\Q)$ and $\dot{\hat \Pp}=\Vp^\top I_t(\Pp)$. The H-OpInf problem of learning the symmetric operator $\Dhatq$ can be written as 
 \begin{align}
  \label{eq:Dhatq}
 &\min_{\Dhatq=\Dhatq^\top} \bigg|\bigg| \dot{\Phat} + \Dhatq\Qhat + \Fhatp (\Qhat, \Phat)  \bigg|\bigg|_F \nonumber \\
 &= \min_{\Dhatq=\Dhatq^\top} \bigg|\bigg| \Vp^\top I_t(\Pp) + \Dhatq\Qhat + \Fhatp (\Qhat, \Phat)  \bigg|\bigg|_F \nonumber \\
&= \min_{\Dhatq=\Dhatq^\top} \bigg|\bigg| \Vp^\top I_t(\Pp) - \Vp^\top \dot{\Pp}_{\rm ana} + \Vp^\top\dot{\Pp}_{\rm ana} + \left(\Dhatq - \tilde{\mathbf D}_{\q}\right) \Qhat +\tilde{\mathbf D}_{\q} \Qhat + \Fhatp (\Qhat, \Phat)  \bigg|\bigg|_F  \nonumber \\
&= \min_{\Dhatq=\Dhatq^\top} \bigg|\bigg| \Vp^\top \left(I_t(\Pp) - \dot{\Pp}_{\rm ana}\right) +  \left(\Dhatq - \tilde{\mathbf D}_{\q}\right) \Qhat + \Vp^\top\dot{\Pp}_{\rm ana} + \tilde{\mathbf D}_{\q} \Qhat + \Fhatp (\Qhat, \Phat)  \bigg|\bigg|_F.   
 \end{align}
 Similarly, the H-OpInf problem of learning the symmetric operator $\Dhatp$ can be written as 
  \begin{align}
  \label{eq:Dhatp}
 &\min_{\Dhatp=\Dhatp^\top} \bigg|\bigg| \dot{\Qhat} - \Dhatp\Phat - \Fhatq (\Qhat, \Phat)  \bigg|\bigg|_F \nonumber \\
&= \min_{\Dhatp=\Dhatp^\top} \bigg|\bigg| \Vq^\top \left(I_t(\Q) - \dot{\Q}_{\rm ana}\right) -  \left(\Dhatp - \tilde{\mathbf D}_{\p}\right) \Phat - \tilde{\mathbf D}_{\p}\Phat+ \Vq^\top\dot{\Q}_{\rm ana} -  \Fhatq (\Qhat, \Phat)  \bigg|\bigg|_F.  
 \end{align}
 Combining \eqref{eq:Dhatq} and \eqref{eq:Dhatp}, the H-OpInf problem can be written as
 \begin{multline}
 \min_{\substack{\Dhatq=\Dhatq^\top,\\ \Dhatp=\Dhatp^\top}} \bigg| \bigg| \begin{bmatrix}
\Vp^{\top} & \bzero \\ \bzero & \Vq^\top 
\end{bmatrix}\underbrace{ \left( \begin{bmatrix}
 \dot{\Pp}_{\rm ana} \\ \dot{\Q}_{\rm ana}
\end{bmatrix}-\begin{bmatrix}
\bzero & -\mathbf{D}_{\q} \\ \mathbf{D}_{\p}  &\bzero  
\end{bmatrix}\begin{bmatrix}
\Vp\Vp^\top\Pp \\ \Vq\Vq^\top\Q
\end{bmatrix}  + \begin{bmatrix}
\Vp\Vq^\top\Fp \\ -\Vq\Vp^\top\Fq
\end{bmatrix}  \right)}_{(I)} \\ + \begin{bmatrix}
\Vp^{\top} & \bzero \\ \bzero & \Vq^\top 
\end{bmatrix}\underbrace{ \begin{bmatrix}
I_t(\Pp) - \dot{\Pp}_{\rm ana} \\ I_t(\Q) - \dot{\Q}_{\rm ana}
\end{bmatrix} }_{(II)} + \begin{bmatrix}
\bzero & \Dhatq-\tilde{\mathbf{D}}_{\q} \\- (\Dhatp-\tilde{\mathbf{D}}_{\p})  &\bzero  
\end{bmatrix}\begin{bmatrix}
\Phat \\ \Qhat
\end{bmatrix}
\bigg| \bigg| _F,
 \end{multline}
 where we have used $\tilde{\mathbf D}_{\q}=\Vq^\top\mathbf D_{\q}\Vq$ and  $\tilde{\mathbf D}_{\p}=\Vp^\top\mathbf D_{\p}\Vp$. Since, we have used the cotangent lift algorithm for projection, we have $\Vq=\Vp=\Phi$. Based on the time discretization assumptions, the terms $(I),(II) \to \bzero$ for $\Delta t \to 0$ and $r \to n$, and thus, the learned operators $\Dhatq,\Dhatp$ converge to the structure-preserving intrusive operators $\tilde{\mathbf D}_{\q},\tilde{\mathbf D}_{\p}$. Therefore, in the pre-asymptotic case, there exists for all $0<\epsilon \in \Real$ a small enough time step $\Delta t$ and a large enough reduced dimension $r \leq n$ such that we can use the full-rank condition of $\Qhat,\Phat$ to deduce $|| \Dhatq- \tilde{\mathbf D}_{\q} ||_F \leq \epsilon, || \Dhatp- \tilde{\mathbf D}_{\p} ||_F \leq \epsilon$.
 \end{proof}
The theorem shows that the learned symmetric Hamiltonian ROM operators converge to the intrusive symmetric Hamiltonian ROM operators as $r \to n$. However, this result does not provide a convergence rate, and in our practical experience there are numerical examples where the difference in reduced operators might not monotonically decrease for low-dimensional Hamiltonian ROMs, such as the linear wave example in Section~\ref{s:42}. In such cases, the FOM energy error becomes a crucial metric to assess the reliability of nonintrusive Hamiltonian ROM for long-time predictions. The continuous-time intrusive Hamiltonian ROM due to its specific choice of the reduced Hamiltonian, i.e., $\tilde{H}(\hat{\y}):=\dH(\V\hat{\y})$, always conserves the FOM Hamiltonian $\dH$.  In contrast, continuous-time nonintrusive Hamiltonian ROM conserves the reduced Hamiltonian $\Hr$. Using the theoretical result, we can interpret the learned reduced Hamiltonian $\Hr$ as a perturbation of the intrusive reduced Hamiltonian $\tilde{H}$, i.e., $\Hr=\tilde{H} + \Delta \tilde{H}$. Since the nonlinear component is the same for both nonintrusive and intrusive Hamiltonian ROMs, the perturbation can be written as 
\begin{equation*}
\Delta \tilde{H}=\frac{1}{2}\qhat^\top (\Dhatq- \tilde{\mathbf D}_{\q} ) \qhat
 + \frac{1}{2}\phat^\top (\Dhatp - \tilde{\mathbf D}_{\p} ) \phat.
\end{equation*}
Thus, the nonintrusive Hamiltonian ROM trajectories simulate a perturbed intrusive Hamiltonian ROM Hamiltonian system and the perturbation $\Delta \tilde{H}$ depends on the difference in reduced operators. 
%
%
\subsection{Computational Procedure}
\label{s:34} 
Due to the canonical nature of the reduced model form, the original optimization problem \eqref{eq:H-OpInf_opt} can be broken down into separate, symmetric linear least-squares problems of the form
\begin{equation}
\min_{\Dhatp=\Dhatp^\top} \bigg| \bigg| 
\dot{\Qhat}- \Fhatq(\Qhat,\Phat) - \Dhatp\Phat \bigg| \bigg| _F, \quad \quad \min_{\Dhatq=\Dhatq^\top} \bigg| \bigg| 
\dot{\Phat} + \Fhatp(\Qhat,\Phat) 
 + \Dhatq\Qhat \bigg| \bigg| _F. 
 \label{eq:HOpInf}
 \end{equation}
Given reduced state data $ \Qhat,\Phat$ and residual data $\hat{\mathbf R}_{\q},\hat{\mathbf R}_{\p}$, our goal is to find symmetric reduced operators $\Dhatq=\Dhatq^\top$ and $\Dhatp=\Dhatp^\top$ that minimize $ \bigg| \bigg| \Qhat^\top\Dhatq - \hat{\mathbf R}_{\q}^\top \bigg| \bigg| _F$ and  $\bigg| \bigg| \Phat^\top\Dhatp - \hat{\mathbf R}_{\p}^\top \bigg| \bigg| _F$ respectively. Both problems are symmetric linear least-squares problems, so let us consider the optimization problem for infering $\Dhatq$. We formulate the symmetric linear least-squares problem 
\begin{equation}
\min_{\Dhatq=\Dhatq^\top}  \bigg| \bigg| \Qhat^\top\Dhatq - \hat{\mathbf R}_{\q}^\top \bigg| \bigg|^2 _F,
\end{equation}
as a constrained optimization problem with the following Lagrangian
\begin{equation}
\mathcal{L}(\Dhatq,\mathbf{\Lambda}):=  \bigg| \bigg| \Qhat^\top\Dhatq - \hat{\mathbf R}_{\q}^\top \bigg| \bigg|^2 _F + \langle \mathbf{\Lambda}, \Dhatq-\Dhatq^\top \rangle,
\end{equation}
where $\mathbf{\Lambda} \in \Real^{r \times r}$ is the Lagrange multiplier and  $\langle \cdot, \cdot \rangle$ is the elementwise inner product, i.e., $\langle \mathbf{\Lambda}, \hat{\mathbf{D}} \rangle:=\text{Tr}(\mathbf{\Lambda}^\top\hat{\mathbf{D}})$.
Differentiating the Lagrangian with respect to $\Dhatq$ and $\mathbf{\Lambda}$, we obtain the following matrix equations
\begin{align*}
2\Qhat(\Qhat^\top\Dhatq- \hat{\mathbf R}_{\q}^\top) + \mathbf{\Lambda} - \mathbf{\Lambda}^\top = \mathbf{0}, \\
\Dhatq - \Dhatq^\top = \mathbf{0},
\end{align*}
where the second equation is simply the symmetric constraint condition. Rewriting the first matrix equation as
\begin{equation}
\Qhat(\Qhat^\top\Dhatq- \hat{\mathbf R}_{\q}^\top) = -\frac{1}{2} \left(  \mathbf{\Lambda} - \mathbf{\Lambda}^\top  \right),
\end{equation}
reveals that $\Qhat(\Qhat^\top\Dhatq- \hat{\mathbf R}_{\q}^\top)$ is a skew-symmetric matrix, i.e., 
\begin{equation}
\Qhat(\Qhat^\top\Dhatq- \hat{\mathbf R}_{\q}^\top) = -\left(\Qhat(\Qhat^\top\Dhatq- \hat{\mathbf R}_{\q}^\top)\right)^\top.
\end{equation}
Rewriting the above equation, we can obtain the following Lyapunov equation
\begin{equation}
(\Qhat\Qhat^\top)\Dhatq + \Dhatq (\Qhat\Qhat^\top)=\Qhat \hat{\mathbf R}_{\q}^\top + \hat{\mathbf R}_{\q}\Qhat^\top.
\label{eq:lyapq}
\end{equation}
Similarly, the symmetric reduced operator $\Dhatp$ solves 
\begin{equation}
(\Phat\Phat^\top)\Dhatp + \Dhatp (\Phat\Phat^\top)=\Phat \hat{\mathbf R}_{\p}^\top + \hat{\mathbf R}_{\p}\Phat^\top.
\label{eq:lyapp}
\end{equation}
Thus, the original inference problem for $\Dhatq$ and $\Dhatq$ can be broken down into separate symmetric linear least-squares problems and subsequently, the symmetric reduced operators can be obtained by solving Lyapunov equations \eqref{eq:lyapq} and \eqref{eq:lyapp}. Algorithm \ref{alg} summarizes H-OpInf for Hamiltonian systems with nonpolynomial nonlinearities. \\
 \begin{algorithm}
\caption{Hamiltonian operator inference for canonical Hamiltonian systems}
\begin{algorithmic}[1]
\Require Nonlinear component of Hamiltonian functional $H_{\rm nl}$, snapshot data $\Q, \Pp \in \Real^{n \times K}$, and reduced dimension $2r$. 
\Ensure Symmetric reduced operators $\Dhatq$ and $\Dhatp$.
    \State Use knowledge of $H_{\rm nl}$ to identify correct model form for the reduced Hamiltonian $\Hr$ \eqref{eq:Hrmodel}.
    \State Compute forcing snapshot data $\Fq,\Fp \in \Real^{n \times K}$ \eqref{eq:F}.
        \State Compute  time-derivative data $\dot{\Q},\dot{\Pp} \in \Real^{n \times K}$ \eqref{eq:dQ} from state trajectory data $\Q,\Pp$ using the finite-difference scheme \eqref{eq:fd}.
    \State Build symplectic basis matrix $\V \in \Real^{2n \times 2r}$ from $\mathbf{M}=[\Q,\Pp]$ using the cotangent lift algorithm, see Section \ref{s:22}.
    \State Project to obtain reduced state data $\Qhat,\Phat \in \Real^{r \times K}$ \eqref{eq:Yhat}, reduced time-derivative data $\dot{\Qhat},\dot{\Phat} \in \Real^{r \times K}$ \eqref{eq:dQhat}, and reduced nonlinear forcing data $\Fhatq, \Fhatp \in \Real^{r \times K}$ \eqref{eq:Fhat}.
    \State Solve separate symmetric linear least-squares problems via Lyapunov equations \eqref{eq:lyapq}--\eqref{eq:lyapp} to nonintrusively infer symmetric reduced operators $\Dhatq$ and $\Dhatp$.
\end{algorithmic}
\label{alg}
\end{algorithm}
%
\section{Numerical Results}
\label{s:4} 
In this section, we study the numerical performance of H-OpInf for three Hamiltonian PDEs with increasing level of complexity. We revisit the linear wave equation from Section \ref{s:31} and demonstrate that H-OpInf works for different structure-preserving space discretizations. We then study the nonlinear Schr\"{o}dinger equation, a Hamiltonian PDE with cubic nonlinearity, to investigate the numerical performance of H-OpInf for nonlinear systems. Finally, we apply our structure-preserving operator inference approach to the sine-Gordon equation to understand its numerical performance for PDEs with nonpolynomial nonlinearities.

\subsection{Numerical Implementation Details}
\label{s:41}
We compare the quality of the learned Hamiltonian ROM with the intrusive proper symplectic decomposition (PSD), a Hamiltonian ROM obtained via the cotangent lift algorithm as outlined in Section~\ref{s:22}. Below we give some information about our FOM and ROM numerical simulations:
\begin{itemize}
\item For numerical time integration, we use the implicit midpoint rule \eqref{eq:symp} for all FOMs and ROMs. 
The implicit midpoint rule is a symplectic scheme for Hamiltonian systems which preserves the symplectic structure and exhibits bounded energy error for both FOM and ROM simulations. The symplectic structure preservation also implies exact preservation of any quadratic invariants of the motion. The implicit midpoint rule satisfies Assumption \ref{as:1}; details about the numerical properties of this symplectic time integrator can be found in \cite{hairer2006geometric,leimkuhler_reich_2005}.
\item To compute time-derivative data from the snapshot data we use the following fourth-order finite difference scheme 
\begin{equation}
\dot{\y}_{k}\approx \frac{-\y_{k+2}+ 8\y_{k+1}-8\y_{k-1}  + \y_{k-2}}{12\dt}.
\label{eq:fd}
\end{equation}
For first two and last two points, we used first-order forward and backward Euler approximations, respectively. The finite difference scheme used in this work for computing time-derivative data satisfies Assumption \ref{as:2}.
\item All numerical examples are computed with MATLAB version 2020b. The Lyapunov equations \eqref{eq:lyapq}--\eqref{eq:lyapp} are solved using the in-built \texttt{lyap} function in MATLAB.
\item All the state error plots in this section compute the following relative error 
\begin{equation}
\frac{|| \Y - \V \hat{\Y} ||_2}{|| \Y||_2}
\label{eq:state_error}
\end{equation}
where $\hat{\Y}$ is either obtained from nonintrusive Hamiltonian ROM or intrusive Hamiltonian ROM. For state approximation error in training data, we only consider trajectories up to the training time interval and for test data plots, we consider trajectories from the full ROM simulation. 
\item All FOM energy error plots in this section compute the following error
\begin{equation}
    |\dH(\V\hat{\y}(t))-\dH(\V\hat{\y}(0)|
    \label{eq:FOM_energy}
\end{equation}
where $\hat{\y}$ is either obtained from nonintrusive Hamiltonian ROM or intrusive Hamiltonian ROM. For ROM energy error plots, we compare $|\hat{H}(\hat{\y}(t))-\hat{H}(\hat{\y}(0)|$ for nonintrusive Hamiltonian ROMs of different sizes.
\end{itemize}
%
%
\subsection{Linear Wave Equation}
\label{s:42} 
\subsubsection{Finite Difference Discretization}
\label{s:421} 
We revisit the linear wave example from Section~\ref{s:31} with finite difference spatial discretization. Using the H-OpInf framework, we infer symmetric reduced operators $\Dhatq=\Dhatq^\top \in \Real^{r \times r}$ for $2r=40$ using Algorithm \ref{alg}. Nonintrusive Hamiltonian ROMs of size $2w$ with $2w<2r$ are constructed by extracting submatrices of size $w \times w$, corresponding to the first $w$ basis vectors, from $\Dhatq$ and $\Dhatp$. Thus, our structure-preserving approach requires performing H-OpInf with Algorithm \ref{alg} only once.
\par
Figure \ref{fig:LW_state} compares the errors of the intrusive Hamiltonian ROMs and nonintrusive Hamiltonian ROMs over the training time interval of $T=10$. Compared with Figure \ref{fig: stacked_state}, we observe that the state errors shown in Figure \ref{fig:LW_state} don't display oscillations. The nonintrusive Hamiltonian ROM shows a similar behavior to the intrusive Hamiltonian ROMs for the linear wave example up to $2r=32$. However, as $2r$ increases further, the error for the nonintrusive Hamiltonian ROMs levels off. This leveling off of the state error is because the projected trajectories correspond to non-Markovian dynamics in the reduced setting even though the FOM state trajectories and the corresponding FOM dynamics are Markovian, see \cite{peherstorfer2020sampling}. Figure \ref{fig:LW_state_test} compares the state approximation error over the testing time interval of $T=100$ where both Hamiltonian ROMs demonstrate similar error up to $2r=20$. For $2r>20$, the intrusive Hamiltonian ROM errors decrease more rapidly compared to nonintrusive Hamiltonian ROM.

\begin{figure}
\captionsetup[subfigure]{oneside,margin={1.5cm,0 cm}}
\begin{subfigure}{.38\textwidth}
       \setlength\fheight{4 cm}
        \setlength\fwidth{\textwidth}
%
%
\definecolor{mycolor1}{rgb}{0.63529,0.07843,0.18431}%
\definecolor{mycolor2}{rgb}{0.85098,0.32549,0.09804}%
\begin{tikzpicture}

\begin{axis}[%
width=0.974\fwidth,
height=\fheight,
at={(0\fwidth,0\fheight)},
scale only axis,
xmin=0,
xmax=40,
xlabel style={font=\color{white!15!black}},
xlabel={\small Reduced Dimension $2r$},
ymode=log,
ymin=0.001,
ymax=1,
yminorticks=true,
ylabel style={font=\color{white!15!black}},
ylabel={\small Relative State Error},
axis background/.style={fill=white},
xmajorgrids,
ymajorgrids,
legend style={at={(0.375,0.75)}, anchor=south west, legend cell align=left, align=left, draw=white!15!black,font=\footnotesize}
]
\addplot [color=mycolor1, dashdotted,line width=1.5pt, mark=square, mark options={solid, mycolor1}]
  table[row sep=crcr]{%
4	0.736141583751316\\
8	0.430061117347179\\
12	0.107562312697577\\
16	0.0150895754869499\\
20	0.0104224960538857\\
24	0.0067147391238826\\
28	0.00364395804636036\\
32	0.00272882303607118\\
36	0.00193314779330995\\
40	0.00170108730762371\\
};
\addlegendentry{H-OpInf}

\addplot [color=mycolor2, dotted, line width=1.5pt, mark=x, mark options={solid, mycolor2}]
  table[row sep=crcr]{%
4	0.736141370847748\\
8	0.430061095428298\\
12	0.107561337272124\\
16	0.0150799740086443\\
20	0.0104010499576645\\
24	0.00666412431497805\\
28	0.00356984905550367\\
32	0.00261280201423283\\
36	0.00172989235307004\\
40	0.00138680720785551\\
};
\addlegendentry{Intrusive PSD}

\end{axis}

\begin{axis}[%
width=1.257\fwidth,
height=1.257\fheight,
at={(-0.163\fwidth,-0.163\fheight)},
scale only axis,
xmin=0,
xmax=1,
ymin=0,
ymax=1,
axis line style={draw=none},
ticks=none,
axis x line*=bottom,
axis y line*=left
]
\end{axis}
\end{tikzpicture}%
\caption{State error (training data)}
\label{fig:LW_state}
    \end{subfigure}
    \hspace{1.4cm}
    \begin{subfigure}{.38\textwidth}
           \setlength\fheight{4 cm}
           \setlength\fwidth{\textwidth}
%
%
\definecolor{mycolor1}{rgb}{0.63529,0.07843,0.18431}%
\definecolor{mycolor2}{rgb}{0.85098,0.32549,0.09804}%
\begin{tikzpicture}
\begin{axis}[%
width=0.974\fwidth,
height=\fheight,
at={(0\fwidth,0\fheight)},
scale only axis,
xmin=0,
xmax=40,
xlabel style={font=\color{white!15!black}},
xlabel={\small Reduced Dimension $2r$},
ymode=log,
ymin=0.001,
ymax=1,
yminorticks=true,
ylabel style={font=\color{white!15!black}},
ylabel={\small Relative State Error},
axis background/.style={fill=white},
xmajorgrids,
ymajorgrids,
legend style={at={(0.372,0.755)}, anchor=south west, legend cell align=left, align=left, draw=white!15!black,font=\footnotesize}
]
\addplot [color=mycolor1, dashdotted,line width=1.5pt, mark=square, mark options={solid, mycolor1}]
  table[row sep=crcr]{%
4	0.736104087878812\\
8	0.430596286600137\\
12	0.107687626899745\\
16	0.015998363244194\\
20	0.0123407494883214\\
24	0.0110601341413955\\
28	0.00809511739423335\\
32	0.00808532355479226\\
36	0.00832065343606134\\
40	0.00835941160374378\\
};
\addlegendentry{H-OpInf}

\addplot [color=mycolor2, dotted, line width=1.5pt, mark=x, mark options={solid, mycolor2}]
  table[row sep=crcr]{%
4	0.736099236579802\\
8	0.430590043267018\\
12	0.107593543546421\\
16	0.0150708081937535\\
20	0.0103970121659926\\
24	0.0078376396564234\\
28	0.00357719944133545\\
32	0.00268040797157253\\
36	0.00173933132807277\\
40	0.00152538843444477\\
};
\addlegendentry{Intrusive PSD}

\end{axis}

\begin{axis}[%
width=1.257\fwidth,
height=1.257\fheight,
at={(-0.163\fwidth,-0.163\fheight)},
scale only axis,
xmin=0,
xmax=1,
ymin=0,
ymax=1,
axis line style={draw=none},
ticks=none,
axis x line*=bottom,
axis y line*=left
]
\end{axis}
\end{tikzpicture}%
\caption{State error (test data)}
\label{fig:LW_state_test}
    \end{subfigure}
\caption{Linear wave equation (finite difference discretization): nonintrusive Hamiltonian ROMs obtained with H-OpInf presented in Section~\ref{s:33} achieve similar state error performance to that of intrusive Hamiltonian ROMs. For test data, the state approximation error for learned Hamiltonian ROMs does not decrease as favorably with increase in reduced dimension for $2r>20$.}
 \label{fig:LW_state_fd}
\end{figure}
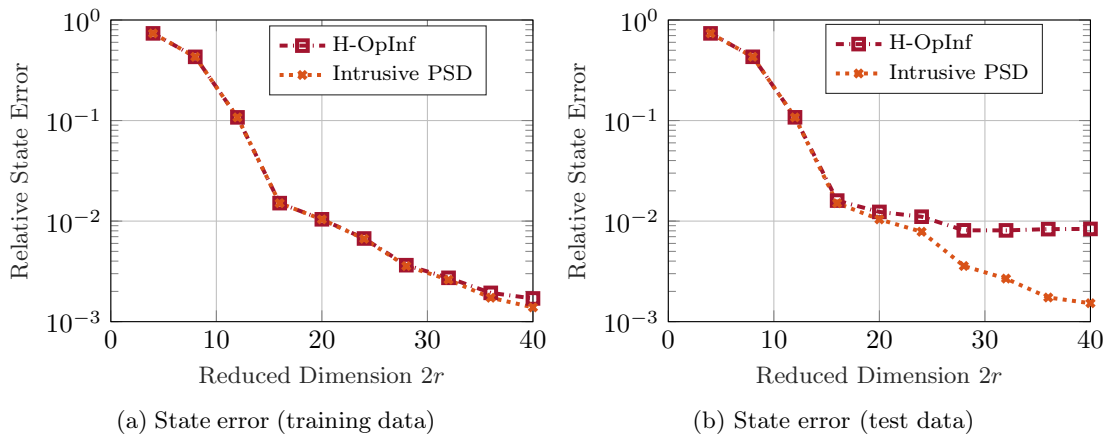

\begin{figure}
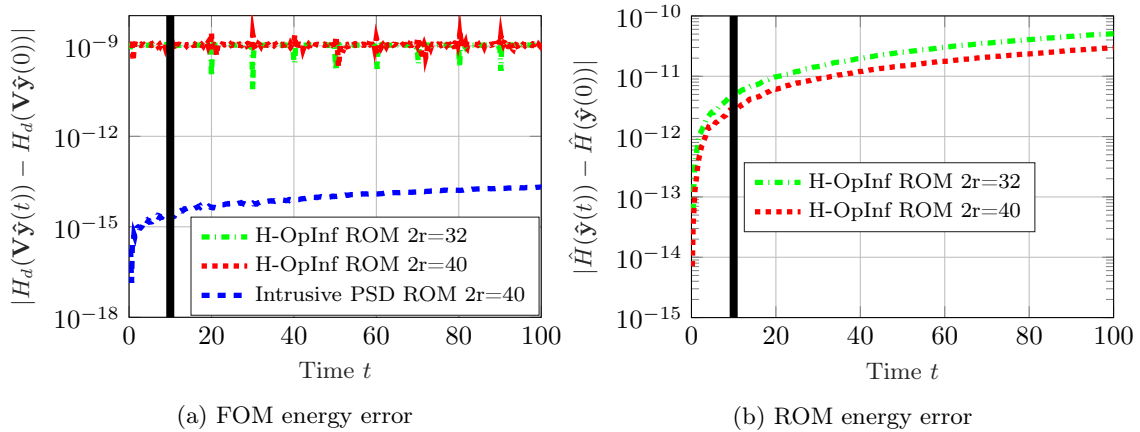

\captionsetup[subfigure]{oneside,margin={2cm,0 cm}}
\begin{subfigure}{.38\textwidth}
       \setlength\fheight{4 cm}
        \setlength\fwidth{\textwidth}
\input{Figures/figures/LW/LW_Hd.tex}
\caption{FOM energy error}
 \label{fig:LW_FOM}
    \end{subfigure}
    \hspace{1.4cm}
    \begin{subfigure}{.38\textwidth}
           \setlength\fheight{4 cm}
           \setlength\fwidth{\textwidth}         
\input{Figures/figures/LW/LW_hrr.tex}
\caption{ROM energy error }
 \label{fig:LW_ROM}
    \end{subfigure}
\caption{Linear wave equation (finite difference discretization): Plot (a) shows bounded FOM energy error around $10^{-9}$ for nonintrusive Hamiltonian ROMs. The intrusive Hamiltonian ROM, on the other hand, shows exact energy preservation due to its specific choice of reduced Hamiltonian. The black line indicates end of training time interval. Plot (b) shows similar ROM energy error for both nonintrusive Hamiltonian ROMs with marginal improvement in the error magnitude for $2r=40$.}
 \label{fig:LW_energy}
\end{figure}

In Figure \ref{fig:LW_FOM}, we compare the FOM energy error for different nonintrusive ROMs. The system is simulated for $T=100$, thus predicting the numerical behavior for $900\%$ past the training interval. Due to its specific choice of reduced Hamiltonian, the intrusive Hamiltonian ROM conserves the energy with the same accuracy as the FOM simulation. The nonintrusive Hamiltonian ROMs exhibit bounded energy error for different ROM sizes due to their Hamiltonian nature. The bounded energy error $900\%$ past the training interval shows that the nonintrusive Hamiltonian ROMs simulate a perturbation of the intrusive Hamiltonian ROM exactly and thus, the FOM energy error for nonintrusive ROMs remains bounded well beyond the training data. This shows a true strength of the proposed H-OpInf, namely that if the structure is respected in every aspect of discretization and the learning method, long-term stable predictions are possible. Figure \ref{fig:LW_ROM} compares ROM energy error for nonintrusive Hamiltonian ROMs where both reduced models demonstrate similar energy error behavior. The bounded energy error plots in Figure~\ref{fig:LW_energy} affirm the Hamiltonian nature of the nonintrusive ROMs for different $r$. Additionally, these nonintrusive ROM energy error plots also serve as an indicator of long-time stability of simulations of our learned ROM.

\begin{remark}
The state error leveling-off for operator inference has been resolved by a sampling scheme based on re-projection in \cite{peherstorfer2020sampling}. If the data is sampled with that scheme, then we recover the intrusive ROMs preasymptotically from data under certain conditions. However, re-projection in its current formulation only works for fully discrete systems with explicit or linearly implicit time-marching schemes. Nonlinear Hamiltonian systems, on the other hand, require fully implicit time integrators to preserve the underlying geometric structure. Thus, the re-projection sampling can not be used for Hamiltonian systems. Extending this algorithm for use with H-OpInf remains a topic of further investigation.
\end{remark}

\subsubsection{Pseudo-spectral Discretization}
\label{s:422}
We now consider the linear wave example with the same setup as in Section \ref{s:31}, except that the linear wave PDE is spatially discretized with the pseudo-spectral method. 
Spatial discretization using the pseudo-spectral method involves two key steps. 
\begin{enumerate}
\item Construct space-discretized representation of the solution through interpolating trigonometric polynomial of the solution at collocation points in the domain.
\item Derive space-discretized FOM equations for the discrete values of the solution  from the PDE by finding an approximation for the differential operator $\partial_{xx}$ in terms of the discrete values from the space-discretized representation. 
\end{enumerate}
The FOM is represented by the following Hamiltonian ODE system
\begin{equation*}
    \dot{\textbf{y}}=\Jn \nabla_{ \textbf{y} }H_{ d}(\textbf{y})=\begin{bmatrix}
       \bzero & \In \\
       -\In & \bzero
     \end{bmatrix} \begin{bmatrix}
       -c^2\mathbf{D}_{\text{ps}} & \bzero \\
       \bzero & \In
     \end{bmatrix}\textbf{y} =\begin{bmatrix}
       \bzero & \In \\
       c^2\mathbf{D}_{\text{ps}} & \bzero
     \end{bmatrix} \textbf{y} ,
\end{equation*}
where $\mathbf{D}_{\text{ps}}$ denotes the pseudo-spectral approximation for $\partial_{xx}$. Note that the FOM equations have the same coupling structure for the pseudo-spectral discretization as in \eqref{eq:LW_FOM} for the finite difference discretization and hence, the same ROM model form will be used for learning the symmetric reduced operators for this case.

The state error plots for the pseudo-spectral discretization case in Figure \ref{fig:LW_pseudo} are similar to the finite difference case. In Figure \ref{fig:LW_pseudo_train}, the state error for the nonintrusive Hamiltonian ROM is nearly same as the intrusive Hamiltonian ROM for $2r\leq30$. However in test data regime, the state error comparison in Figure \ref{fig:LW_pseudo_test} show that the state error for nonintrusive Hamiltonian ROM does not decrease as favorably after $2r=20$, similar to the case in Section \ref{s:421}. In the FOM energy error plots in Figure \ref{fig:LW_pseudo_Hd}, we observe bounded energy error for $900\%$ past the training time interval which shows the stability of the learned Hamiltonian ROMs. The ROM energy error comparison in Figure \ref{fig:LW_pseudo_Hr} shows the energy error eventually stabilizes around $10^{-10}$ for both nonintrusive Hamiltonian ROMs. The ability of H-OpInf to handle different structure-preserving spatial discretizations is the key takeaway from this study. Figure \ref{fig:LW_traj} compares the state vector solution at $t=7$, $t=41.5$, and $t=100$. Both intrusive and nonintrusive Hamiltonian ROMs accurately capture the solution profile, well beyond the training data regime. These results demonstrate the ability of the learned Hamiltonian ROM to provide reliable long-time predictions.

\begin{figure}
\captionsetup[subfigure]{oneside,margin={1.5cm,0 cm}}
\begin{subfigure}{.38\textwidth}
       \setlength\fheight{4 cm}
        \setlength\fwidth{\textwidth}
%
%
\definecolor{mycolor1}{rgb}{0.63529,0.07843,0.18431}%
\definecolor{mycolor2}{rgb}{0.85098,0.32549,0.09804}%
\begin{tikzpicture}

\begin{axis}[%
width=0.974\fwidth,
height=\fheight,
at={(0\fwidth,0\fheight)},
scale only axis,
xmin=0.0268817204301053,
xmax=40.0268817204301,
xlabel style={font=\color{white!15!black}},
xlabel={\small Reduced Dimension $2r$},
ymode=log,
ymin=0.001,
ymax=1,
yminorticks=true,
ylabel style={font=\color{white!15!black}},
ylabel={\small Relative State Error},
axis background/.style={fill=white},
xmajorgrids,
ymajorgrids,
legend style={at={(0.341,0.642)}, anchor=south west, legend cell align=left, align=left, draw=white!15!black,font=\footnotesize}
]
\addplot [color=mycolor1, dashdotted,line width=1.5pt, mark=square, mark options={solid, mycolor1}]
  table[row sep=crcr]{%
4	0.736153490267292\\
8	0.430079819267268\\
12	0.10757447712037\\
16	0.0151011121333136\\
20	0.0104303578832355\\
24	0.00665670971243451\\
28	0.00365759435228332\\
32	0.00273477212695122\\
36	0.00194887400049198\\
40	0.00165805856696717\\
};
\addlegendentry{H-OpInf}

\addplot [color=mycolor2, dotted, line width=1.5pt, mark=x, mark options={solid, mycolor2}]
  table[row sep=crcr]{%
4	0.736153269156566\\
8	0.430079815439364\\
12	0.107573527667914\\
16	0.0150914913322116\\
20	0.0104087547436695\\
24	0.0066169135444112\\
28	0.0035816989611912\\
32	0.00261559038167936\\
36	0.00174088963370737\\
40	0.00135736395732388\\
};
\addlegendentry{Intrusive PSD}

\end{axis}

\begin{axis}[%
width=1.257\fwidth,
height=1.257\fheight,
at={(-0.163\fwidth,-0.163\fheight)},
scale only axis,
xmin=0,
xmax=1,
ymin=0,
ymax=1,
axis line style={draw=none},
ticks=none,
axis x line*=bottom,
axis y line*=left
]
\end{axis}
\end{tikzpicture}%
\caption{State error (training data)}
 \label{fig:LW_pseudo_train}
    \end{subfigure}
    \hspace{1.4cm}
    \begin{subfigure}{.38\textwidth}
           \setlength\fheight{4 cm}
           \setlength\fwidth{\textwidth}
%
%
\definecolor{mycolor1}{rgb}{0.63529,0.07843,0.18431}%
\definecolor{mycolor2}{rgb}{0.85098,0.32549,0.09804}%
\begin{tikzpicture}

\begin{axis}[%
width=0.974\fwidth,
height=\fheight,
at={(0\fwidth,0\fheight)},
scale only axis,
xmin=0,
xmax=40,
xlabel style={font=\color{white!15!black}},
xlabel={\small Reduced Dimension $2r$},
ymode=log,
ymin=0.001,
ymax=1,
yminorticks=true,
ylabel style={font=\color{white!15!black}},
ylabel={\small Relative State Error},
axis background/.style={fill=white},
xmajorgrids,
ymajorgrids,
legend style={at={(0.297,0.738)}, anchor=south west, legend cell align=left, align=left, draw=white!15!black,font=\footnotesize}
]
\addplot [color=mycolor1, dashdotted,line width=1.5pt, mark=square, mark options={solid, mycolor1}]
  table[row sep=crcr]{%
4	0.736117592580442\\
8	0.431017184693097\\
12	0.107702182652236\\
16	0.0160248006146913\\
20	0.0123658493935255\\
24	0.00984025001290515\\
28	0.00818262858752999\\
32	0.00816096305652389\\
36	0.00838970488418691\\
40	0.00843601656701927\\
};
\addlegendentry{H-OpInf}

\addplot [color=mycolor2, dotted, line width=1.5pt, mark=x, mark options={solid, mycolor2}]
  table[row sep=crcr]{%
4	0.73611221438343\\
8	0.430863205082545\\
12	0.107607898223832\\
16	0.0150950884935087\\
20	0.0104108685398884\\
24	0.00663514852484815\\
28	0.00358067812877152\\
32	0.00261500795489447\\
36	0.00174035142090659\\
40	0.00139437977459913\\
};
\addlegendentry{Intrusive PSD}

\end{axis}

\begin{axis}[%
width=1.257\fwidth,
height=1.257\fheight,
at={(-0.163\fwidth,-0.163\fheight)},
scale only axis,
xmin=0,
xmax=1,
ymin=0,
ymax=1,
axis line style={draw=none},
ticks=none,
axis x line*=bottom,
axis y line*=left
]
\end{axis}
\end{tikzpicture}%
\caption{State error (test data)}
 \label{fig:LW_pseudo_test}
    \end{subfigure}
\caption{Linear wave equation (pseudo-spectral discretization): Plots show that H-OpInf is able to produce accurate nonintrusive Hamiltonian ROMs even if the PDE is discretized with pseudo-spectral method in space.}
 \label{fig:LW_pseudo}
\end{figure}
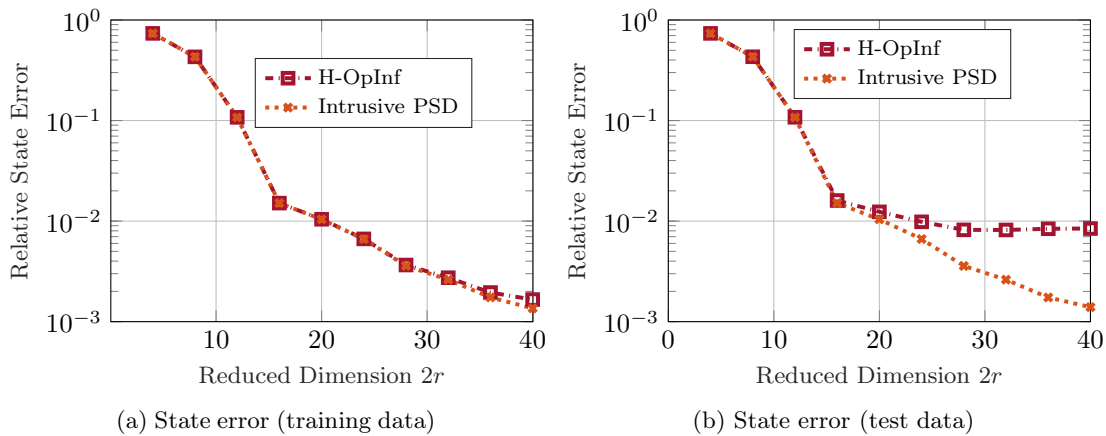

\begin{figure}
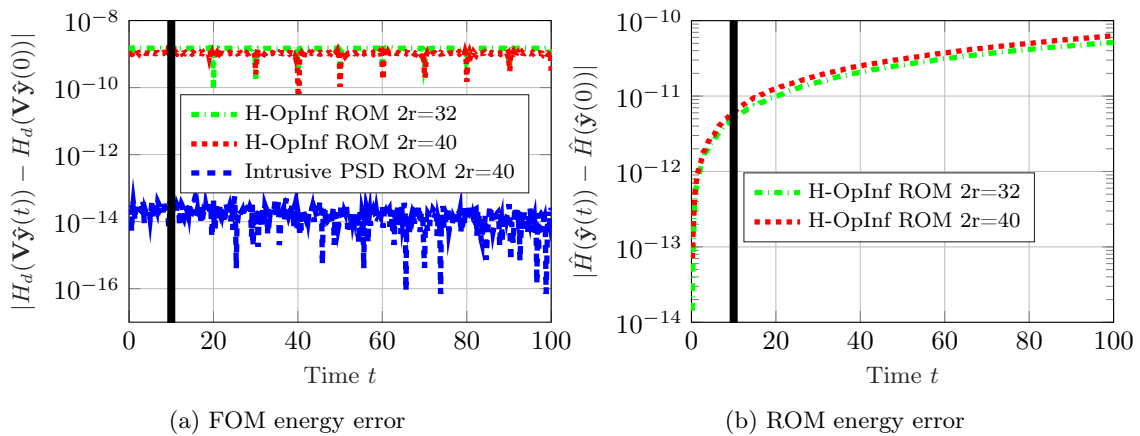

\captionsetup[subfigure]{oneside,margin={1.8cm,0 cm}}
\begin{subfigure}{.38\textwidth}
       \setlength\fheight{4 cm}
        \setlength\fwidth{\textwidth}
\input{Figures/figures/LW/LW_pseudo_hdd.tex}
\caption{FOM energy error}
 \label{fig:LW_pseudo_Hd}
    \end{subfigure}
    \hspace{1.4cm}
    \begin{subfigure}{.38\textwidth}
           \setlength\fheight{4 cm}
           \setlength\fwidth{\textwidth}
\input{Figures/figures/LW/LW_pseudo_hrr.tex}
\caption{ROM energy error }
 \label{fig:LW_pseudo_Hr}
    \end{subfigure}
\caption{Linear wave equation (pseudo-spectral discretization): Plot (a) shows bounded FOM energy error for both nonintrusive Hamiltonian ROMs with marginally lower error for $2r=40$. The black line indicates end of training time interval. Plot (b) shows similar ROM energy error for nonintrusive ROMs.}
 \label{fig:LW_pseudo_energy}
\end{figure}


\begin{figure}
\captionsetup[subfigure]{oneside,margin={1.8cm,0 cm}}
\begin{subfigure}{.25\textwidth}
       \setlength\fheight{3 cm}
        \setlength\fwidth{\textwidth}
\raisebox{-51mm}{\input{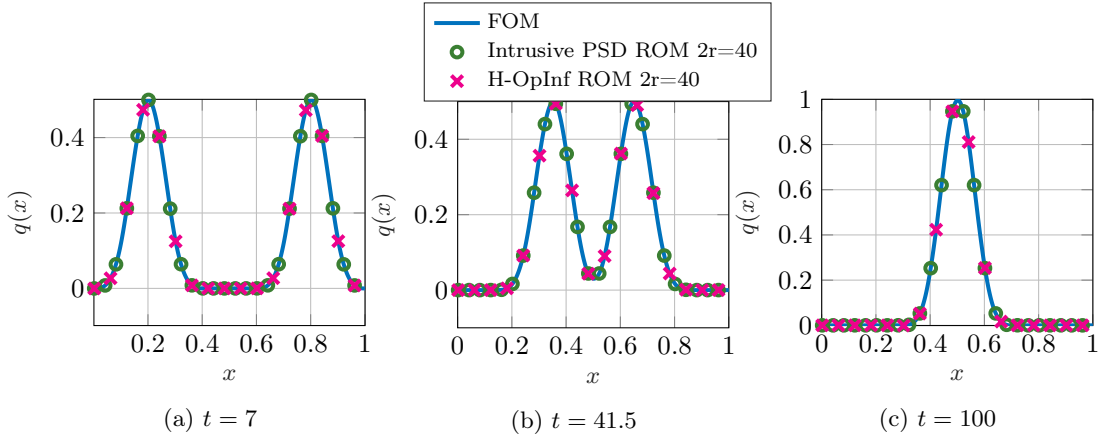}}
\subcaption{$t=7$ }
    \end{subfigure}
   \hspace{0.8cm}
    \begin{subfigure}{.25\textwidth}
           \setlength\fheight{3 cm}
           \setlength\fwidth{\textwidth}
%
%
\definecolor{mycolor1}{rgb}{0.00000,0.44700,0.74100}%
\definecolor{mycolor2}{rgb}{0.85000,0.32500,0.09800}%
\definecolor{mycolor3}{rgb}{0.92900,0.69400,0.12500}%
\begin{tikzpicture}

\begin{axis}[%
width=0.951\fwidth,
height=\fheight,
at={(0\fwidth,0\fheight)},
scale only axis,
xmin=0,
xmax=1,
xlabel style={font=\color{white!15!black}},
xlabel={\small $x$},
ymin=-0.1,
ymax=0.5,
ylabel style={font=\color{white!15!black}},
ylabel={ \small $q(x)$},
axis background/.style={fill=white},
xmajorgrids,
ymajorgrids,
legend style={at={(-0.125,1)}, anchor=south west, legend cell align=left, align=left, draw=white!15!black,font=\footnotesize}
]
\addplot [color=mycolor1,line width=1.5 pt]
  table[row sep=crcr]{%
0.002	-4.45465552267876e-06\\
0.0660000000000001	-1.11058444995926e-06\\
0.138	4.59245564587984e-05\\
0.148	0.000171315950925344\\
0.154	0.000319895968608597\\
0.16	0.000600360763952645\\
0.164	0.000895561283250634\\
0.168	0.00130706570766215\\
0.172	0.00186972481562475\\
0.176	0.00263572822207236\\
0.178	0.00311172877248467\\
0.18	0.00365719659090957\\
0.182	0.00427808298148746\\
0.184	0.00498202709241591\\
0.186	0.0057779541272005\\
0.188	0.00667404090132262\\
0.19	0.00767812776598342\\
0.192	0.00879611093738086\\
0.194	0.0100342514978422\\
0.196	0.0114004208366019\\
0.198	0.0129019730760547\\
0.2	0.0145466523610553\\
0.202	0.0163438778382636\\
0.204	0.0182988452556136\\
0.206	0.0204121009278488\\
0.208	0.0226869703913377\\
0.21	0.0251295408005987\\
0.212	0.0277478916018816\\
0.214	0.0305481017079217\\
0.216	0.0335358593568058\\
0.218	0.0367154030750689\\
0.22	0.0400898385371125\\
0.222	0.0436611491042709\\
0.224	0.0474345162502712\\
0.226	0.0514171337323346\\
0.228	0.0556154785567589\\
0.23	0.0600367999948168\\
0.232	0.0646880200124356\\
0.234	0.0695766064476289\\
0.236	0.0747071302962174\\
0.238	0.0800808350538589\\
0.24	0.0856981430934254\\
0.242	0.0915603108469787\\
0.244	0.0976737169193649\\
0.246	0.104045625239801\\
0.248	0.110679428824641\\
0.25	0.117576235033136\\
0.252	0.124733696405859\\
0.254	0.132146696447929\\
0.258	0.147715728510862\\
0.262	0.164231641900501\\
0.266	0.181642042031121\\
0.27	0.199882106643531\\
0.274	0.218843557813987\\
0.278	0.23837189688711\\
0.284	0.268350282125583\\
0.294	0.318560128920801\\
0.298	0.338173984598325\\
0.302	0.357259732401369\\
0.306	0.375669596437436\\
0.31	0.3932897947618\\
0.314	0.409996805317738\\
0.316	0.417971045879204\\
0.318	0.425674956042231\\
0.32	0.433095013949342\\
0.322	0.440217618635127\\
0.324	0.447026105415486\\
0.326	0.453502306273931\\
0.328	0.459627916791305\\
0.33	0.465383703045043\\
0.332	0.470750977505181\\
0.334	0.475710528676395\\
0.336	0.480243463336486\\
0.338	0.484333617689919\\
0.34	0.487965433665964\\
0.342	0.491123443938937\\
0.344	0.493788230346478\\
0.346	0.49594212912108\\
0.348	0.497572802649048\\
0.35	0.498671967762019\\
0.352	0.499236240218769\\
0.354	0.49926623655108\\
0.356	0.498764834883454\\
0.358	0.497736633682589\\
0.36	0.496185275524734\\
0.362	0.494112763069726\\
0.364	0.491519167657522\\
0.366	0.488404032628351\\
0.368	0.484767397906323\\
0.37	0.480613676726809\\
0.372	0.475952778909208\\
0.374	0.470800677014863\\
0.376	0.465177159367282\\
0.378	0.459105430436779\\
0.38	0.452610740395824\\
0.382	0.445722210467351\\
0.384	0.438471131264372\\
0.386	0.430885304395621\\
0.388	0.422990774189769\\
0.392	0.406390065660694\\
0.396	0.388887850095115\\
0.4	0.370670455238662\\
0.404	0.351899812127683\\
0.41	0.322985216185677\\
0.418	0.283573735811282\\
0.43	0.224201970307628\\
0.436	0.195249755872118\\
0.44	0.176592650132735\\
0.444	0.158685336491611\\
0.448	0.141709157955012\\
0.45	0.133606934724946\\
0.452	0.125769915621076\\
0.454	0.118200047212754\\
0.456	0.110898104098334\\
0.458	0.103864372485537\\
0.46	0.0971035972300798\\
0.462	0.0906247618988263\\
0.464	0.0844406404220139\\
0.466	0.0785672880066117\\
0.468	0.0730224033816473\\
0.47	0.0678211791055121\\
0.472	0.0629740048511191\\
0.474	0.0584874550160051\\
0.476	0.0543625174682034\\
0.478	0.0505952561150975\\
0.48	0.0471795737834788\\
0.482	0.0441055638216314\\
0.484	0.0413609676212336\\
0.486	0.0389361769947316\\
0.488	0.0368230875202444\\
0.49	0.0350122180601364\\
0.492	0.0334944566317859\\
0.494	0.0322629938839405\\
0.496	0.0313114629598068\\
0.498	0.0306355382576866\\
0.5	0.0302327310652117\\
0.502	0.0300993388685031\\
0.504	0.0302327310674868\\
0.506	0.0306355382482739\\
0.508	0.0313114629598641\\
0.51	0.0322629938820613\\
0.512	0.0334944566292883\\
0.514	0.0350122180563595\\
0.516	0.0368230875182753\\
0.518	0.0389361769980276\\
0.52	0.041360967619962\\
0.522	0.0441055638255603\\
0.524	0.047179573779035\\
0.526	0.0505952561191965\\
0.528	0.0543625174603897\\
0.53	0.058487455013722\\
0.532	0.0629740048556025\\
0.534	0.0678211791023642\\
0.536	0.0730224033788902\\
0.538	0.0785672880113537\\
0.54	0.0844406404160623\\
0.542	0.0906247618908309\\
0.544	0.0971035972329528\\
0.546	0.103864372479766\\
0.548	0.110898104094512\\
0.55	0.118200047205483\\
0.552	0.125769915613555\\
0.554	0.13360693472301\\
0.558	0.15007215414054\\
0.562	0.167532561975221\\
0.566	0.185840287287033\\
0.57	0.204795529774983\\
0.576	0.234018761317031\\
0.598	0.342349330296619\\
0.602	0.361344685078606\\
0.606	0.379858364509984\\
0.61	0.397738921977024\\
0.614	0.414816277284258\\
0.616	0.422990774191146\\
0.618	0.430885304401925\\
0.62	0.438471131261796\\
0.622	0.445722210469583\\
0.624	0.452610740396408\\
0.626	0.45910543043934\\
0.628	0.46517715936958\\
0.63	0.470800677019582\\
0.632	0.475952778911257\\
0.634	0.480613676725908\\
0.636	0.48476739791168\\
0.638	0.488404032632923\\
0.64	0.491519167654712\\
0.642	0.494112763071596\\
0.644	0.496185275524227\\
0.646	0.497736633680916\\
0.648	0.498764834891234\\
0.65	0.499266236552846\\
0.652	0.499236240219337\\
0.654	0.498671967755456\\
0.656	0.497572802641944\\
0.658	0.495942129127369\\
0.66	0.493788230343967\\
0.662	0.491123443941228\\
0.664	0.487965433664965\\
0.666	0.484333617686451\\
0.668	0.480243463338374\\
0.67	0.475710528678956\\
0.672	0.470750977506297\\
0.674	0.465383703046514\\
0.676	0.459627916787919\\
0.678	0.453502306270693\\
0.68	0.447026105409411\\
0.682	0.440217618633517\\
0.684	0.433095013948542\\
0.686	0.425674956034125\\
0.688	0.417971045870691\\
0.69	0.409996805316114\\
0.694	0.393289794760632\\
0.698	0.375669596429987\\
0.702	0.357259732395542\\
0.706	0.338173984597526\\
0.71	0.31856012891985\\
0.718	0.278430682975277\\
0.726	0.238371896886147\\
0.73	0.218843557811645\\
0.734	0.19988210664405\\
0.738	0.181642042025895\\
0.742	0.164231641903346\\
0.746	0.147715728506104\\
0.75	0.132146696445414\\
0.752	0.124733696403231\\
0.754	0.117576235036299\\
0.756	0.110679428820842\\
0.758	0.104045625236786\\
0.76	0.0976737169184609\\
0.762	0.0915603108420668\\
0.764	0.085698143100279\\
0.766	0.0800808350539257\\
0.768	0.074707130293878\\
0.77	0.0695766064562953\\
0.772	0.0646880200093565\\
0.774	0.0600367999983087\\
0.776	0.0556154785580372\\
0.778	0.0514171337346701\\
0.78	0.0474345162515595\\
0.782	0.0436611491105874\\
0.784	0.0400898385382951\\
0.786	0.0367154030711971\\
0.788	0.0335358593608939\\
0.79	0.0305481017078573\\
0.792	0.0277478916087366\\
0.794	0.0251295407938179\\
0.796	0.0226869703916455\\
0.798	0.0204121009297431\\
0.8	0.0182988452617825\\
0.802	0.016343877847222\\
0.804	0.0145466523586251\\
0.806	0.0129019730768265\\
0.808	0.011400420842427\\
0.81	0.0100342515008209\\
0.812	0.00879611094221788\\
0.814	0.007678127771928\\
0.816	0.00667404090931445\\
0.818	0.00577795413176907\\
0.82	0.00498202709457884\\
0.822	0.00427808298833354\\
0.824	0.00365719659317021\\
0.826	0.00311172877413979\\
0.83	0.00222386765292293\\
0.834	0.00156678078188155\\
0.838	0.00108490664180461\\
0.842	0.000735168804374853\\
0.846	0.0004881183715848\\
0.852	0.000259215256032785\\
0.862	7.93595209676035e-05\\
0.872	2.92065145492959e-05\\
0.898	-3.75669562235892e-06\\
0.934	9.6202467769757e-08\\
1	-2.443334000235e-06\\
};
\addlegendentry{FOM}

\addplot [color=OliveGreen, only marks, line width=1.5pt, mark size=2pt,mark=o, mark options={solid, OliveGreen}]
  table[row sep=crcr]{%
0.002	-2.47884094873863e-05\\
0.042	2.25326041918272e-05\\
0.082	-4.02541618242935e-05\\
0.122	1.34583284687739e-05\\
0.162	0.000791672006668631\\
0.202	0.0163010658545134\\
0.242	0.0915532540893963\\
0.282	0.258315126373318\\
0.322	0.440209679158841\\
0.362	0.494099937716344\\
0.402	0.361346621810089\\
0.442	0.167608113297041\\
0.482	0.0440782743788165\\
0.522	0.044078274377325\\
0.562	0.167608113295633\\
0.602	0.361346621806567\\
0.642	0.49409993771551\\
0.682	0.440209679158034\\
0.722	0.258315126372708\\
0.762	0.0915532540887692\\
0.802	0.0163010658517336\\
0.842	0.00079167200489727\\
0.882	1.34583274689071e-05\\
0.922	-4.02541635500242e-05\\
0.962	2.25326025518058e-05\\
};
\addlegendentry{Intrusive PSD ROM 2r=40}

\addplot [color=magenta, only marks,line width=1.5pt, mark size=3pt, mark=x, mark options={solid, magenta}]
  table[row sep=crcr]{%
0.002	-2.78062168823645e-06\\
0.0620000000000001	1.61071080801989e-05\\
0.122	1.83376673889013e-05\\
0.182	0.00460877506015966\\
0.242	0.09165567523556\\
0.302	0.356493922828862\\
0.362	0.4943392307183\\
0.422	0.26439235551337\\
0.482	0.0434350165923812\\
0.542	0.0903093113977355\\
0.602	0.362153223589279\\
0.662	0.490508721842606\\
0.722	0.257277943720139\\
0.782	0.0439135367777508\\
0.842	0.000931616031092153\\
0.902	-1.43399962572177e-05\\
0.962	4.926980308817e-05\\
};
\addlegendentry{H-OpInf ROM 2r=40}

\end{axis}

\begin{axis}[%
width=1.227\fwidth,
height=1.227\fheight,
at={(-0.16\fwidth,-0.135\fheight)},
scale only axis,
xmin=0,
xmax=1,
ymin=0,
ymax=1,
axis line style={draw=none},
ticks=none,
axis x line*=bottom,
axis y line*=left
]
\end{axis}
\end{tikzpicture}%
\subcaption{$t=41.5$}
    \end{subfigure}
    \hspace{0.8 cm}
        \begin{subfigure}{.25\textwidth}
           \setlength\fheight{3 cm}
           \setlength\fwidth{\textwidth}
\raisebox{-51mm}{
%
%
\definecolor{mycolor1}{rgb}{0.00000,0.44700,0.74100}%
\definecolor{mycolor2}{rgb}{0.85000,0.32500,0.09800}%
\definecolor{mycolor3}{rgb}{0.92941,0.69412,0.12549}%
\begin{tikzpicture}

\begin{axis}[%
width=0.951\fwidth,
height=\fheight,
at={(0\fwidth,0\fheight)},
scale only axis,
xmin=0,
xmax=1,
xlabel style={font=\color{white!15!black}},
xlabel={\small $x$},
ymin=-0.00101561880334378,
ymax=1,
ylabel style={font=\color{white!15!black}},
ylabel={\small $q(x)$},
axis background/.style={fill=white},
xmajorgrids,
ymajorgrids
]
\addplot [color=mycolor1,line width=1.5pt]
  table[row sep=crcr]{%
0.002	-1.42927299113804e-05\\
0.0820000000000001	-1.59096070795339e-05\\
0.0920000000000001	9.58681220097901e-06\\
0.1	1.97440683704642e-05\\
0.112	-3.43686525718923e-05\\
0.126	4.58246371248805e-05\\
0.14	-6.83070866276392e-05\\
0.162	4.2817819976726e-05\\
0.176	-6.87308499522921e-05\\
0.194	3.36843014641097e-05\\
0.21	7.10984143972304e-05\\
0.224	-0.000143527430585655\\
0.234	-2.69444692670984e-05\\
0.242	3.15703710012105e-05\\
0.256	2.31463789890363e-05\\
0.262	0.000182201829593831\\
0.272	0.000484765091325112\\
0.276	0.000481500833263482\\
0.28	0.000378702321230051\\
0.284	0.000169147548188864\\
0.29	-0.000289794519419129\\
0.296	-0.000728983604105116\\
0.3	-0.000903850937605588\\
0.304	-0.000920728461146103\\
0.306	-0.000850344279618165\\
0.308	-0.000715886735828342\\
0.31	-0.000508981718135404\\
0.312	-0.000226059201110962\\
0.314	0.00013487612699592\\
0.316	0.000574130893770119\\
0.318	0.00109170515957935\\
0.32	0.00169175534001509\\
0.322	0.00238037955929005\\
0.324	0.00316567270850876\\
0.326	0.00405780240740161\\
0.328	0.00506901622960765\\
0.33	0.00621236719328455\\
0.332	0.00749882174420091\\
0.334	0.00893829399528645\\
0.336	0.0105429673211723\\
0.338	0.0123276639766321\\
0.34	0.0143087592749727\\
0.342	0.0165005974273544\\
0.344	0.0189149233362267\\
0.346	0.0215644016065752\\
0.348	0.0244604016258407\\
0.35	0.0276169590292972\\
0.352	0.0310490789567464\\
0.354	0.0347698548671354\\
0.356	0.0387915612063239\\
0.358	0.0431300586631056\\
0.36	0.0478027984751317\\
0.362	0.0528305627116377\\
0.364	0.0582363669415873\\
0.366	0.0640441350931031\\
0.368	0.0702758942562165\\
0.37	0.0769536321503652\\
0.372	0.0840976133380784\\
0.374	0.0917254703526016\\
0.376	0.0998519606843347\\
0.378	0.108488738831698\\
0.38	0.117647622673587\\
0.382	0.127338002783798\\
0.384	0.137565787446372\\
0.386	0.148334134020098\\
0.388	0.159640765726282\\
0.39	0.17147845459915\\
0.392	0.183834984017808\\
0.394	0.196696439470325\\
0.396	0.210052444080707\\
0.398	0.223895670725228\\
0.4	0.238216453103429\\
0.404	0.268240185692477\\
0.408	0.30001358769538\\
0.412	0.333413293277636\\
0.416	0.368324806335609\\
0.42	0.404633960859771\\
0.424	0.442187239600672\\
0.428	0.480780000210616\\
0.434	0.540089624577797\\
0.448	0.680054344606253\\
0.452	0.7189157076217\\
0.456	0.75659848616103\\
0.46	0.792724121907824\\
0.464	0.82690137833357\\
0.466	0.843136549582234\\
0.468	0.858739792921089\\
0.47	0.873666330379462\\
0.472	0.88787407120254\\
0.474	0.901324165398034\\
0.476	0.91398141103093\\
0.478	0.925813711897728\\
0.48	0.936790679516392\\
0.482	0.946884347806888\\
0.484	0.956069594107251\\
0.486	0.9643251770946\\
0.488	0.971635117062165\\
0.49	0.9779886129341\\
0.492	0.983379001442884\\
0.494	0.9878009247568\\
0.496	0.991248771730142\\
0.498	0.993716286470883\\
0.5	0.995198829463382\\
0.502	0.99569339403726\\
0.504	0.995198829457118\\
0.506	0.993716286468935\\
0.508	0.991248771730699\\
0.51	0.987800924756266\\
0.512	0.983379001442114\\
0.514	0.977988612926051\\
0.516	0.971635117064366\\
0.518	0.964325177103412\\
0.52	0.95606959410969\\
0.522	0.946884347807633\\
0.524	0.93679067951929\\
0.526	0.925813711905679\\
0.528	0.913981411027877\\
0.53	0.901324165400457\\
0.532	0.887874071200792\\
0.534	0.873666330377258\\
0.536	0.858739792924805\\
0.538	0.843136549588101\\
0.54	0.826901378335536\\
0.542	0.810081070093362\\
0.546	0.774880006587961\\
0.55	0.737927846412358\\
0.554	0.699608802521814\\
0.56	0.640390982848841\\
0.576	0.480780000214736\\
0.58	0.442187239605853\\
0.584	0.404633960869684\\
0.588	0.368324806352103\\
0.592	0.333413293274371\\
0.596	0.300013587693166\\
0.6	0.268240185689619\\
0.604	0.238216453103385\\
0.606	0.223895670716779\\
0.608	0.210052444080678\\
0.61	0.196696439464922\\
0.612	0.18383498400996\\
0.614	0.171478454605793\\
0.616	0.159640765735159\\
0.618	0.148334134027222\\
0.62	0.137565787451076\\
0.622	0.127338002786078\\
0.624	0.117647622678417\\
0.626	0.108488738837275\\
0.628	0.099851960681729\\
0.63	0.0917254703697565\\
0.632	0.0840976133419165\\
0.634	0.0769536321532636\\
0.636	0.0702758942482511\\
0.638	0.064044135084792\\
0.64	0.0582363669537376\\
0.642	0.0528305627105718\\
0.644	0.0478027984736678\\
0.646	0.0431300586607211\\
0.648	0.0387915612103831\\
0.65	0.0347698548696211\\
0.652	0.031049078952029\\
0.654	0.0276169590263493\\
0.656	0.024460401621591\\
0.658	0.0215644016066747\\
0.66	0.0189149233330062\\
0.662	0.0165005974393826\\
0.664	0.0143087592797673\\
0.666	0.0123276639756891\\
0.668	0.010542967327243\\
0.67	0.00893829399058044\\
0.672	0.00749882174535399\\
0.674	0.00621236719346174\\
0.676	0.00506901622339462\\
0.678	0.00405780240662734\\
0.68	0.00316567270445489\\
0.682	0.00238037955578707\\
0.684	0.00169175533955657\\
0.686	0.00109170515434753\\
0.688	0.000574130905107051\\
0.69	0.000134876128015993\\
0.692	-0.000226059206025697\\
0.694	-0.000508981721830226\\
0.696	-0.000715886741249783\\
0.698	-0.000850344281770221\\
0.702	-0.000935826192373312\\
0.706	-0.000831766912211762\\
0.71	-0.000601326146967729\\
0.724	0.000378702317243462\\
0.728	0.000481500830084025\\
0.732	0.000484765087430006\\
0.738	0.000327978096055004\\
0.748	2.31463773765483e-05\\
0.754	-6.77663366732695e-06\\
0.768	-5.56947931551477e-07\\
0.782	-0.000141745237317359\\
0.788	-5.42885673482374e-05\\
0.796	9.44514769278815e-05\\
0.804	7.51660522420039e-05\\
0.828	-6.87308589868429e-05\\
0.836	-3.11252517364302e-05\\
0.848	7.3351193880189e-05\\
0.856	1.77152902145483e-05\\
0.866	-7.01404178895526e-05\\
0.876	2.30815761033121e-05\\
0.882	6.14091801827144e-05\\
0.902	4.2062827789735e-06\\
0.91	2.37678032068978e-05\\
0.922	-1.5909612993692e-05\\
0.942	-3.63012471682467e-06\\
0.984	-6.413066982347e-06\\
0.994	8.78331140929944e-06\\
1	-1.2805777326319e-05\\
};

\addplot [color=OliveGreen, only marks, line width=1.5pt, mark size=2pt,mark=o, mark options={solid, OliveGreen}]
  table[row sep=crcr]{%
0.002	8.24726835892964e-05\\
0.042	-6.10067156190386e-05\\
0.082	-1.18007731355618e-05\\
0.122	7.57433388677287e-05\\
0.162	3.9985074635096e-05\\
0.202	0.000118704746789877\\
0.242	-6.92466549101489e-07\\
0.282	0.000288208606869511\\
0.322	0.0023709212026235\\
0.362	0.0527884450095474\\
0.402	0.252985533063144\\
0.442	0.620458920391467\\
0.482	0.946874556533289\\
0.522	0.946874556534154\\
0.562	0.620458920392526\\
0.602	0.252985533064428\\
0.642	0.0527884450078367\\
0.682	0.00237092119851623\\
0.722	0.00028820860401424\\
0.762	-6.92468510310462e-07\\
0.802	0.000118704745299514\\
0.842	3.99850729221329e-05\\
0.882	7.57433389138029e-05\\
0.922	-1.18007720529834e-05\\
0.962	-6.10067135533576e-05\\
};

\addplot [color=magenta, only marks, line width=1.5pt, mark size=3pt,mark=x, mark options={solid, magenta}]
  table[row sep=crcr]{%
0.002	2.46017099089224e-05\\
0.0620000000000001	-1.01191098940134e-05\\
0.122	-4.59829194542349e-05\\
0.182	7.83537655024702e-05\\
0.242	-0.00016197618572511\\
0.302	-0.00101561880334378\\
0.362	0.0518462384238465\\
0.422	0.42315947057732\\
0.482	0.946527563536019\\
0.542	0.810987413190171\\
0.602	0.253701804287212\\
0.662	0.0166511133910044\\
0.722	-0.000178702224963523\\
0.782	6.62003994933436e-06\\
0.842	-0.00013775136779115\\
0.902	8.98404683060683e-05\\
0.962	-5.96384329919486e-05\\
};

\end{axis}
\end{tikzpicture}%
}
\subcaption{$t=100$}
    \end{subfigure}
    \caption{Linear wave equation (pseudo-spectral discretization): Plots show the numerical approximation of the linear wave equation using low-dimensional ($2r=40$) intrusive and nonintrusive Hamiltonian ROM at different $t$ values. Despite higher state approximation error compared to intrusive Hamiltonian ROM in the test data regime, the nonintrusive Hamiltonian ROM captures the correct wave form at $t=100$, which is remarkable because after $t=10$ the Hamiltonian ROM simulations are purely predictive. }
\label{fig:LW_traj}
\end{figure}

\subsection{Nonlinear Schr\"{o}dinger Equation}
\label{s:43} 
The nonlinear Schr\"odinger equation (NLSE) is one of the most important integrable Hamiltonian PDEs, and it is used in a wide variety of wave phenomena in physics, including nonlinear optics, gravity waves, Langmuir waves, quantum mechanics, and condensed matter physics. The one-dimensional NLSE considered here is a nonlinear variation of the Schr\"odinger equation and is one of the universal equations that describe the evolution of slowly varying wave packets in weakly nonlinear media with dispersion.
\subsubsection{PDE Formulation}
We consider the cubic Schr\"{o}dinger equation
\begin{equation}
i \psi_t + \psi_{xx} + \gamma |\psi|^2\psi =0,
\end{equation}
where $\psi$ is a complex-valued wave function and $i=\sqrt{-1}$ is the imaginary unit. Writing the complex-valued wave function in terms of its real and imaginary parts as $\psi=p + iq$, yields 
\begin{align*}  
p_t&=-q_{xx}-\gamma(q^2 + p^2)q,  \\
 \quad q_t&= p_{xx}+\gamma(q^2 + p^2)p. 
\end{align*}
The coupled PDEs admit a canonical Hamiltonian PDE form $y_t=J \frac{\delta \mathcal H}{\delta y}$ for $y=[p,q]^\top$ with the following space-time continuous Hamiltonian $\mathcal{H}$
\begin{equation*}  
 \mathcal{H}(q,p)=\int \frac{1}{2} \left[ p_x^2 + q_x^2 -\frac{\gamma}{2}[q^2 + p^2]^2 \right] \ \text{d}x .
\end{equation*}
In addition to the energy conservation, the NLSE also possesses quadratic mass and momentum invariants
\begin{equation*}  
 \mathcal{M}_1(q,p)=\int \left[ p^2 + q^2 \right] \ \text{d}x,  \quad \quad  \mathcal{M}_2(q,p)=\int \left[ p_xq - q_xp  \right] \ \text{d}x. 
\end{equation*}
\subsubsection{FOM Implementation}
We consider the nonlinear Schr\"odinger equation over $x\in [-L/2,L/2]$ with $L=2\sqrt{2}\pi$ and $\gamma=2$. The boundary conditions are periodic and the initial conditions are $p(x,0)=0.5 \left(1 + 0.01\cos(2\pi x/L) \right)$ and $q(x,0)=0$. The nonlinear PDE is spatially discretized using $n=64$ equally spaced grid points leading to a discretized state $\y \in \Real^{128}$. We employ finite difference for spatial discretization to obtain the FOM Hamiltonian
\begin{equation*}
    \dH( \y )=\frac{1}{2}\sum^n_{i=1} \left[ \left( \frac{q_{i+1}-q_i}{\Delta x} \right)^2 +\left( \frac{p_{i+1}-p_i}{\Delta x} \right)^2 - \frac{\gamma}{2}\left( q^2_{i} + p^2_{i}  \right)^2  \right] \ \Delta x,
\end{equation*}
where $p_i:=p(t,x_i)$, $q_i:=q(t,x_i)$, and $ \textbf{y}=[\p^\top \ \q^\top]^\top$. The FOM is represented by the following Hamiltonian ODE system
\begin{equation*} \resizebox{.9\hsize}{!}{$
    \dot{\p}=\nabla_{\q}\dH(\p,\q)
    =-\mathbf D_{\text{fd}}\q - \gamma \begin{bmatrix} \vdots \\  p_{i}^2q_{i} + q_{i}^3 \\ \vdots \end{bmatrix}, \quad
    \dot{\q}=-\nabla_{\p}\dH(\p,\q)
    =\mathbf D_{\text{fd}}\p + \gamma \begin{bmatrix} \vdots \\ p_{i}^3 + q_{i}^2p_{i} \\ \vdots \end{bmatrix}$}.
\end{equation*}
The space-discretized NLSE system conserves the following mass and momentum invariants 
\begin{equation*}
    M_{1,d}( \y )=\sum^n_{i=1} \left[  q^2_{i} + p^2_{i} \right] \ \Delta x, \quad \quad M_{2,d}( \y )=\sum^n_{i=1} \left[  \left(\frac{p_{i+1}-p_i}{\Delta x}\right) q_{i} -\left(\frac{q_{i+1}-q_i}{\Delta x}\right) p_{i} \right] \ \Delta x.
\end{equation*}
The FOM is numerically integrated using the symplectic midpoint rule with $\dt=0.005$. The resulting time-marching equations require solving a system of $2n=128$ coupled nonlinear equations at every time step. We build the snapshot matrix by collecting the snapshots for total time $T=20$. 
\subsubsection{Results}
In Figure \ref{fig:NLSE_state}, we compare the numerical performance of intrusive and nonintrusive Hamiltonian ROMs. The state error plots in Figure \ref{fig:NLSE_state_train} over the training time interval of $T=20$ show that the learned ROM performs similar to the intrusive Hamiltonian ROM up to $2r=10$. For $2r>10$, we observe the leveling-off for the learned ROM where the state accuracy does not improve with increasing reduced dimension. For the test data regime of $T=100$, both intrusive and nonintrusive Hamiltonian ROMs have the same state error up to $2r=10$ in Figure \ref{fig:NLSE_state_test}. For $2r>10$, intrusive Hamiltonian ROM exhibits significantly lower state error compared to nonintrusive Hamiltonian ROM. 
\par
The FOM energy error plots in Figure \ref{fig:NLSE_Hd} show that both intrusive and nonintrusive ROMs have similar bounded energy behavior. Unlike the exact energy conservation for the linear wave equation, the symplectic integrator for NLSE exhibits bounded energy error and hence, nonintrusive and intrusive Hamiltonian ROMs have the same level of FOM energy accuracy. Both Hamiltonian ROMs are simulated until $T=100$, which is $400\%$ longer than the training interval. Due to the nonlinear nature of the problem, both nonintrusive Hamiltonian ROMs in Figure \ref{fig:NLSE_Hr} also exhibit the same energy accuracy.
Given the coupled nature of the governing nonlinear PDE, the competitiveness of the learned ROM with the intrusive Hamiltonian ROM outside the training data is remarkable. 
\subsubsection{Conservation of Invariants for NLSE} 
In quantum mechanics, the quantity $|\psi(x,t)|^2$ represents the probability of finding the system in state $x$ at time $t$. We have compared $|\psi(x,t)|^2$ approximation from the nonintrusive Hamiltonian ROM with the FOM simulation in Figure \ref{fig:NLSE_psi}. The learned Hamiltonian ROM captures the correct distribution of $|\psi(x,t)|^2$, even $400\%$ outside the training time interval. In Figure \ref{fig:NLSE_inv}, we compare the conservation of mass and momentum invariants for the two Hamiltonian ROMs. Due to the specific choice of the reduced Hamiltonian, the intrusive Hamiltonian ROM preserves both quadratic invariants exactly. The learned Hamiltonian ROM exhibits bounded error for the mass invariant and interestingly, conserves the momentum invariant exactly. The numerical results in Figure \ref{fig:NLSE_psi} and Figure \ref{fig:NLSE_inv} demonstrate the physics-preserving nature of our H-OpInf method.  
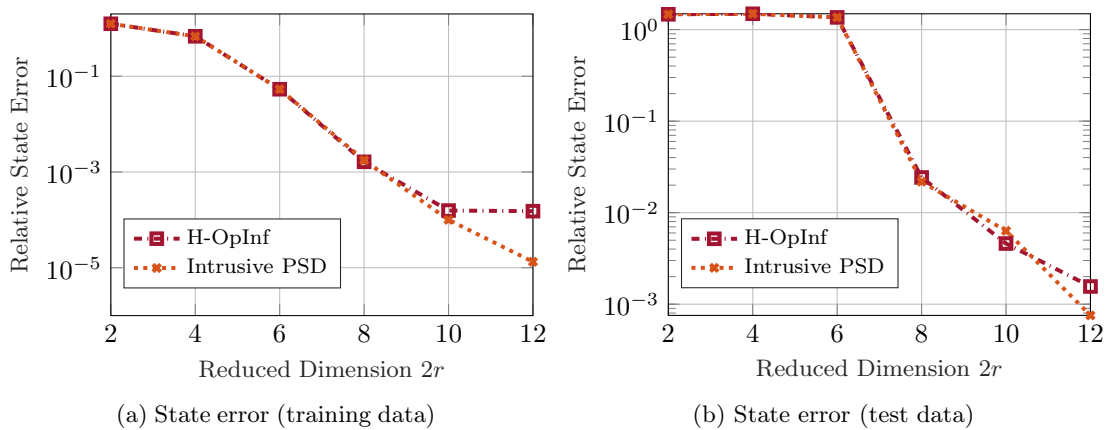
\begin{figure}
\captionsetup[subfigure]{oneside,margin={1.5cm,0 cm}}
\begin{subfigure}{.38\textwidth}
       \setlength\fheight{4 cm}
        \setlength\fwidth{\textwidth}
%
%
\definecolor{mycolor1}{rgb}{0.63529,0.07843,0.18431}%
\definecolor{mycolor2}{rgb}{0.85098,0.32549,0.09804}%
\begin{tikzpicture}

\begin{axis}[%
width=0.974\fwidth,
height=\fheight,
at={(0\fwidth,0\fheight)},
scale only axis,
xmin=2,
xmax=12,
xlabel style={font=\color{white!15!black}},
xlabel={\small Reduced Dimension $2r$},
ymode=log,
ymin=1e-06,
ymax=2,
yminorticks=true,
ylabel style={font=\color{white!15!black}},
ylabel={\small Relative State Error},
axis background/.style={fill=white},
xmajorgrids,
ymajorgrids,
legend style={at={(0.03,0.09)}, anchor=south west, legend cell align=left, align=left, draw=white!15!black, font=\footnotesize}
]
\addplot [color=mycolor1, dashdotted, line width=1.5pt, mark=square, mark options={solid, mycolor1}]
  table[row sep=crcr]{%
2	1.24803762268553\\
4	0.682682794903732\\
6	0.0535034873214138\\
8	0.0016321807456205\\
10	0.000156425485171693\\
12	0.000151930407446278\\
13	0.000151682543848979\\
};
\addlegendentry{H-OpInf}

\addplot [color=mycolor2, dotted, line width=1.5pt, mark=x, mark options={solid, mycolor2}]
  table[row sep=crcr]{%
2	1.24803019150904\\
4	0.682707354622605\\
6	0.053639421061906\\
8	0.00176182301549545\\
10	0.000100323512159385\\
12	1.32972126962278e-05\\
13	3.75677007849872e-06\\
};
\addlegendentry{Intrusive PSD}

\end{axis}

\begin{axis}[%
width=1.257\fwidth,
height=1.257\fheight,
at={(-0.163\fwidth,-0.163\fheight)},
scale only axis,
xmin=0,
xmax=1,
ymin=0,
ymax=1,
axis line style={draw=none},
ticks=none,
axis x line*=bottom,
axis y line*=left
]
\end{axis}
\end{tikzpicture}%
\caption{State error (training data)}
 \label{fig:NLSE_state_train}
    \end{subfigure}
    \hspace{1.4cm}
    \begin{subfigure}{.38\textwidth}
           \setlength\fheight{4 cm}
           \setlength\fwidth{\textwidth}
%
%
\definecolor{mycolor1}{rgb}{0.63529,0.07843,0.18431}%
\definecolor{mycolor2}{rgb}{0.85098,0.32549,0.09804}%
\begin{tikzpicture}

\begin{axis}[%
width=0.974\fwidth,
height=\fheight,
at={(0\fwidth,0\fheight)},
scale only axis,
xmin=2,
xmax=12,
xlabel style={font=\color{white!15!black}},
xlabel={\small Reduced Dimension $2r$},
ymode=log,
ymin=0.000753789267061471,
ymax=1.49086402645261,
yminorticks=true,
ylabel style={font=\color{white!15!black}},
ylabel={\small Relative State Error},
axis background/.style={fill=white},
xmajorgrids,
ymajorgrids,
legend style={at={(0.03,0.09)}, anchor=south west, legend cell align=left, align=left, draw=white!15!black, font=\footnotesize}
]
\addplot [color=mycolor1, dashdotted, line width=1.5pt, mark=square, mark options={solid, mycolor1}]
  table[row sep=crcr]{%
2	1.47173155547955\\
4	1.4908483224346\\
6	1.36158707636561\\
8	0.0242426328597051\\
10	0.00459976017194144\\
12	0.0015604354127782\\
13	0.00187272278450917\\
};
\addlegendentry{H-OpInf}

\addplot [color=mycolor2, dotted, line width=1.5pt, mark=x, mark options={solid, mycolor2}]
  table[row sep=crcr]{%
2	1.47172700611802\\
4	1.49086402645261\\
6	1.36168438597776\\
8	0.0217887032279236\\
10	0.00633106526552354\\
12	0.000753789267061471\\
12.3286308990219	0.000352883328161994\\
};
\addlegendentry{Intrusive PSD}

\end{axis}

\begin{axis}[%
width=1.257\fwidth,
height=1.257\fheight,
at={(-0.163\fwidth,-0.163\fheight)},
scale only axis,
xmin=0,
xmax=1,
ymin=0,
ymax=1,
axis line style={draw=none},
ticks=none,
axis x line*=bottom,
axis y line*=left
]
\end{axis}
\end{tikzpicture}%
\caption{State error (test data)}
 \label{fig:NLSE_state_test}
    \end{subfigure}
\caption{Nonlinear Schr\"odinger equation: The nonintrusive Hamiltonian ROM shows a similar behavior to the H-OpInf ROM for both training and test data in this example. For $2r>10$, we observe a leveling off of the state error in the training data regime for the nonintrusive Hamiltonian ROM.}
 \label{fig:NLSE_state}
\end{figure}
\begin{figure}
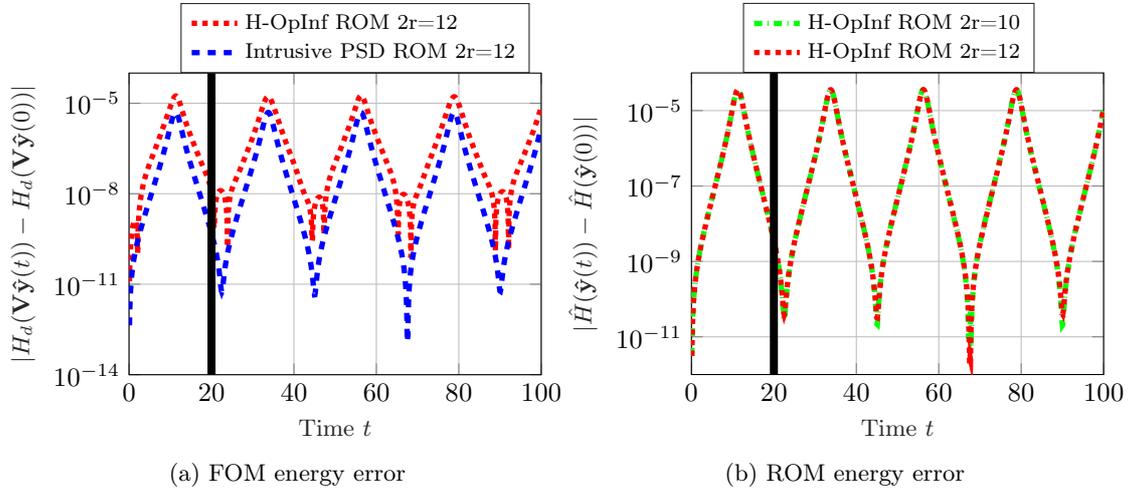

\captionsetup[subfigure]{oneside,margin={1.8cm,0 cm}}
\begin{subfigure}{.38\textwidth}
       \setlength\fheight{4 cm}
        \setlength\fwidth{\textwidth}
\input{Figures/figures/NLSE/NLSE_Hd.tex}
\caption{FOM energy error}
 \label{fig:NLSE_Hd}
    \end{subfigure}
    \hspace{1.4cm}
    \begin{subfigure}{.38\textwidth}
           \setlength\fheight{4 cm}
           \setlength\fwidth{\textwidth}
          
\input{Figures/figures/NLSE/NLSE_Hr.tex}
\caption{ROM energy error }
 \label{fig:NLSE_Hr}
    \end{subfigure}
\caption{Nonlinear Schr\"odinger equation: Plot (a) shows that both learned Hamiltonian ROM and intrusive Hamiltonian ROM exhibit bounded FOM energy error with marginally higher energy error for the nonintrusive Hamiltonian ROM. The black line indicates end of training time interval. Plot (b) shows similar energy error accuracy for nonintrusive Hamiltonian ROMs with $2r=10$ and $2r=12$.}
 \label{fig:NLSE_energy}
\end{figure}
\begin{figure}
\centering
\begin{subfigure}{.45\textwidth}
 \includegraphics[width=\textwidth,height=6cm]{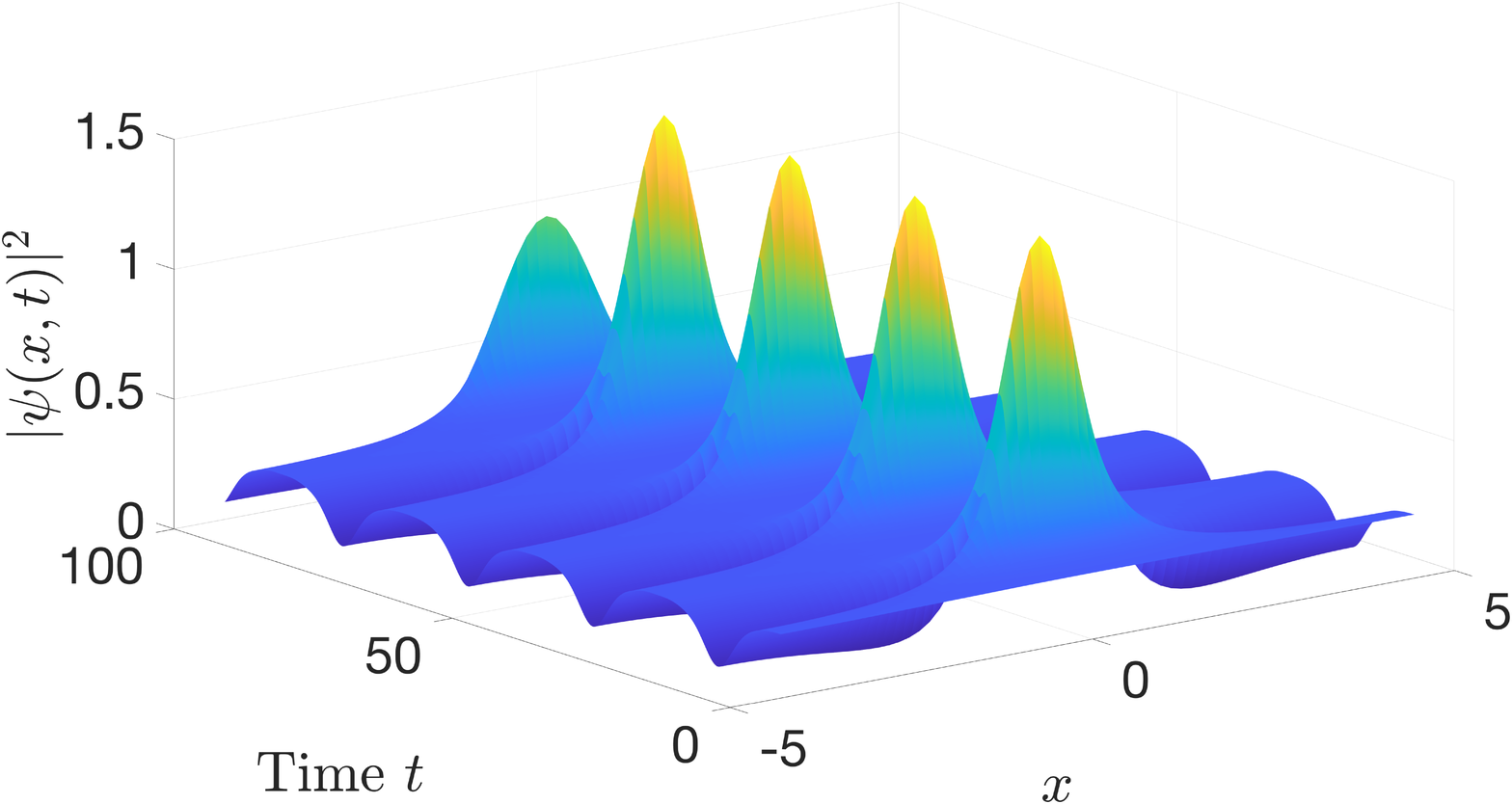}
\caption{FOM}
    \end{subfigure}
    \begin{subfigure}{.45\textwidth}
 \includegraphics[width=\textwidth,height=6cm]{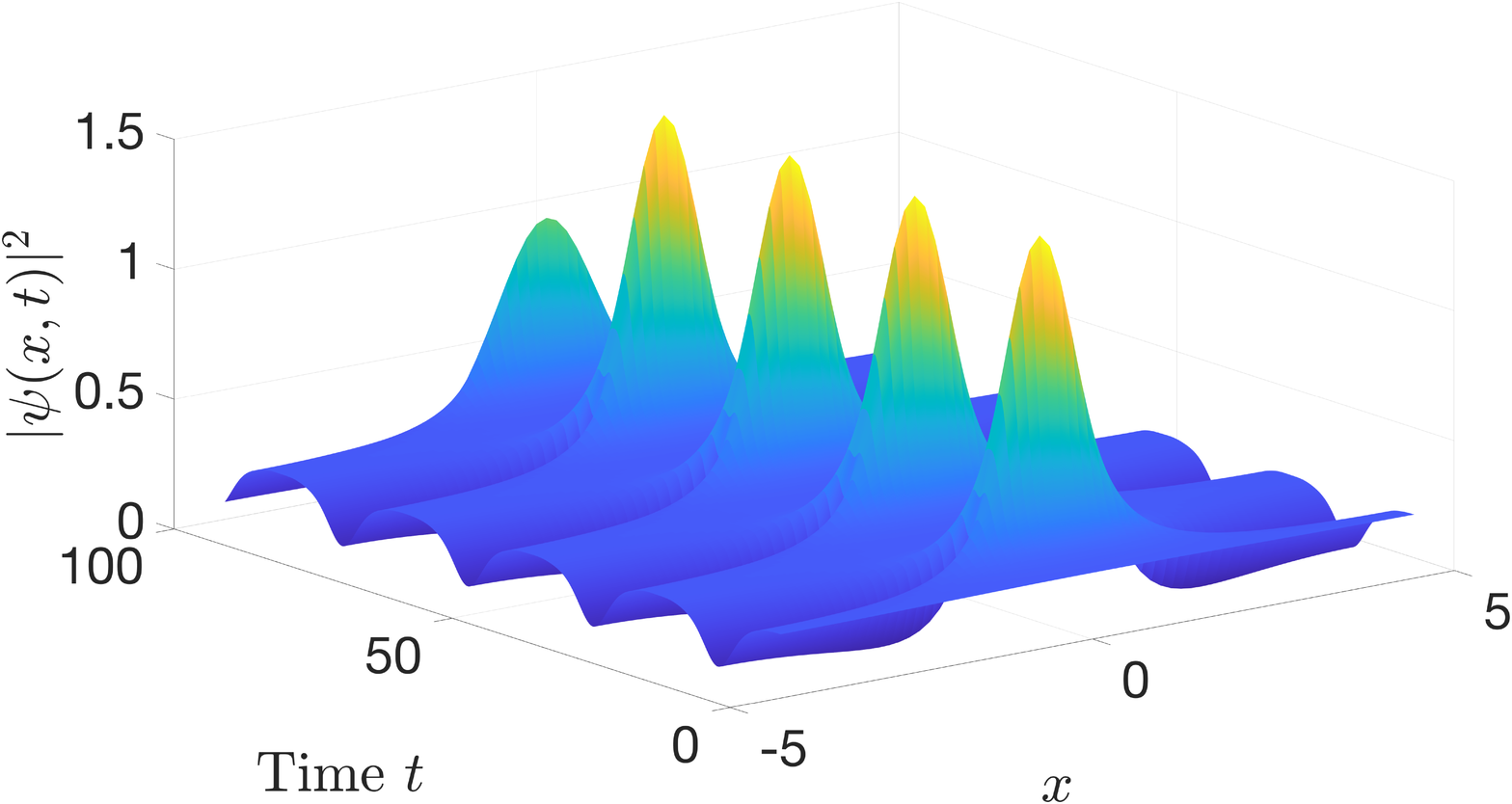}
\caption{H-OpInf ROM $2r=12$ }
    \end{subfigure}
\caption{Nonlinear Schr\"odinger equation:  Hamiltonian ROM learned using H-OpInf correctly predicts $|\psi(x,t)|^2$ well outside the training time interval of $T=20$. }
 \label{fig:NLSE_psi}
\end{figure}
\begin{figure}
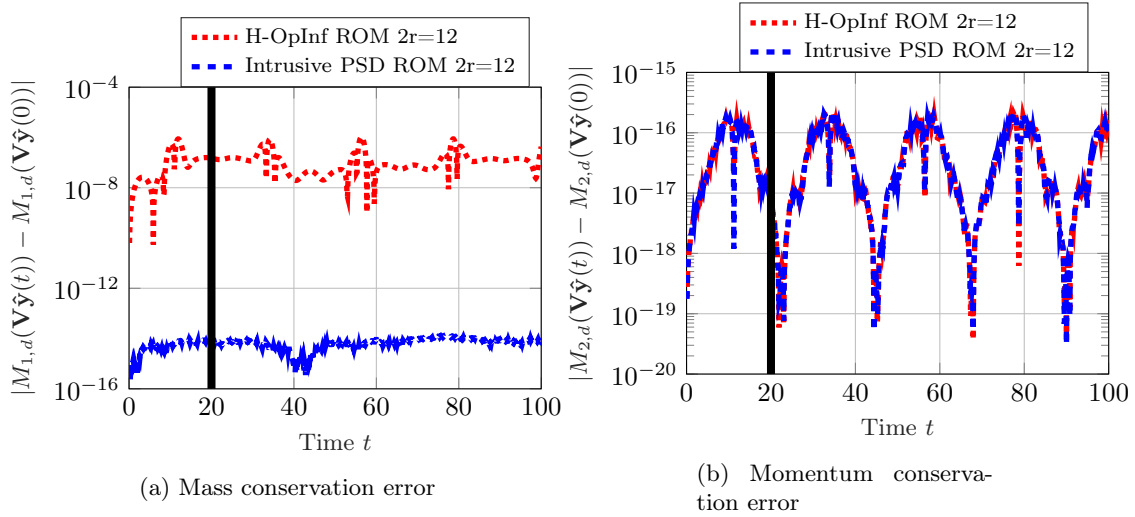

\captionsetup[subfigure]{oneside,margin={1.8cm,0 cm}}
\begin{subfigure}{.38\textwidth}
       \setlength\fheight{4 cm}
        \setlength\fwidth{\textwidth}
\input{Figures/figures/NLSE/NLSE_M1.tex}
\caption{Mass conservation error  }
 \label{fig:NLSE_M1}
    \end{subfigure}
    \hspace{1.4cm}
    \begin{subfigure}{.38\textwidth}
           \setlength\fheight{4 cm}
           \setlength\fwidth{\textwidth}
\input{Figures/figures/NLSE/NLSE_mom.tex}
\caption{Momentum conservation error}
 \label{fig:NLSE_M2}
    \end{subfigure}
\caption{Nonlinear Schr\"odinger equation:  The H-OpInf ROM exhibits bounded error for the mass invariant whereas the intrusive Hamiltonian ROM demonstrates exact conservation. In contrast, both nonintrusive and intrusive Hamiltonian ROMs conserve the momentum invariant exactly. The black line indicates end of training time interval.}
 \label{fig:NLSE_inv}
\end{figure}

\subsection{Sine-Gordon Equation}
\label{s:44} 
We consider a special nonlinear wave equation, called the sine-Gordon equation. This nonlinear hyperbolic PDE appears in a number of physical applications such as Josephson junctions between superconductors, dislocations in crystals, relativistic field theory, and mechanical transmission lines. Although this equation was originally introduced in the 19th century, it came to prominence in 1970s due to the presence of soliton solutions.  A soliton solution is a self-reinforcing wave that maintains its shape while it propagates at a constant velocity in the medium. These special wave forms are caused by a cancellation of nonlinear and dispersive effects in the medium.

\subsubsection{PDE Formulation}
The sine-Gordon equation
\begin{equation}
 \varphi_{tt} = \varphi_{xx} - \sin (\varphi),
 \label{eq:sg}
\end{equation}
has a canonical Hamiltonian formulation for $q=\varphi$ and $p=\varphi_t$. The Hamiltonian functional 
\begin{equation*}  
 \mathcal{H}(q,p)=\int \frac{1}{2} \left[ p^2 + q_x^2 +(1-\cos (q))\right] \ \text{d}x, 
\end{equation*}
with the canonical Hamiltonian PDE form $y_t=J \frac{\delta \mathcal H}{\delta y}$ for $y=[q \ p]^\top$ leads to
\begin{equation*}  y_t=
\begin{bmatrix}
q_t \\p_t
\end{bmatrix}= \begin{bmatrix}
0 & 1 \\-1 & 0
\end{bmatrix}\begin{bmatrix}
\frac{\delta \mathcal H }{\delta q} \\ \frac{\delta \mathcal H}{\delta p}
\end{bmatrix}
= \begin{bmatrix}
p \\
q_{xx} - \sin(q)
\end{bmatrix}.
\end{equation*}

\subsubsection{FOM Implementation}
We study the sine-Gordon equation over $x\in [-L/2,L/2]$ with $L=40$. For the FOM simulations, we consider periodic boundary conditions with the following initial conditions 
\begin{equation}
q(x,0)=0, \quad \quad p(x,0)=\frac{4}{\cosh(x)}.
\end{equation}
The nonlinear PDE is spatially discretized using $n=200$ equally spaced grid points leading to a discretized state $\y \in \Real^{400}$. We discretize the Hamiltonian functional which yields the following space-discretized Hamiltonian $\dH$
\begin{equation*}
\dH( \y)=\sum^n_{i=1} \left[\frac{1}{2} \left( \frac{q_{i+1}-q_i}{\Delta x} \right)^2 +\frac{p_i^2}{2} + (1-\cos(q_i))  \right] \ \Delta x,
\end{equation*}
where $q_i:=q(t,x_i)$, $p_i:=p(t,x_i)$, and $ \y=[\q^\top \ \p^\top]^\top$.
The resulting FOM is represented by the following Hamiltonian ODE system
\begin{equation*} 
    \dot{\q}=\nabla_{\p}\dH(\p,\q)
    =\p, \quad \quad
    \dot{\p}=-\nabla_{\q}\dH(\p,\q)
    =\mathbf D_{\text{fd}}\q - \begin{bmatrix} \vdots \\ \sin (q_i)\\ \vdots \end{bmatrix}.
\end{equation*}
The FOM is numerically integrated for total time $T=10$ using symplectic midpoint rule with $\dt=0.005$. The resulting time-marching equations require solving a system of $2n=400$ coupled nonlinear equations at every time step.

\subsubsection{Results}
Figure \ref{fig:SG_state} shows relative state approximation error for both Hamiltonian ROMs with increasing ROM order. For the training time interval of $T=10$ in Figure \ref{fig:SG_state_train}, both intrusive and nonintrusive approaches yield ROMs of comparative accuracy up to $2r=40$. For $2r>40$, the state error for nonintrusive Hamiltonian ROM levels off. Interestingly for the testing time interval of $T=50$ in Figure \ref{fig:SG_state_test}, the nonintrusive approach gives marginally lower state error for $2r>40$.
\par 
For the FOM energy error comparison, both Hamiltonian ROMs are simulated until $t=400$ to demonstrate the long-time stability of nonintrusive ROM simulations. Unlike the other two numerical examples, we have not considered a full cycle of FOM data for training data in this example. Inside the training data regime, both intrusive and nonintrusive Hamiltonian ROMs produce bounded energy error behavior with similar accuracy in Figure \ref{fig:SG_Hd}. Interestingly, the nonintrusive Hamiltonian FOM energy error plot changes its qualitative behavior after leaving the training data regime but the error still remains bounded up to $t=400$, which is $3900\%$ outside the training interval. Despite the fact that nonintrusive reduced operators are learned only from training data, the Hamiltonian nature of our learned ROM ensures accurate prediction along with bounded energy error far outside the training data regime. Similar to the NLSE example, both nonintrusive Hamiltonian ROMs for the sine-Gordon equation also exhibit the same ROM energy error behavior in Figure \ref{fig:SG_Hr}.
\par
We have compared the approximate numerical solution using both intrusive and nonintrusive approaches with the FOM solution in Figure \ref{fig:SG_traj}. Even though the reduced operators are learned from training data $T=10$, our nonintrusive Hamiltonian ROM captures the correct wave shape at $t=50$ which is $400\%$ past the training time interval. 
\begin{figure}
\captionsetup[subfigure]{oneside,margin={1.5cm,0 cm}}
\begin{subfigure}{.38\textwidth}
       \setlength\fheight{4 cm}
        \setlength\fwidth{\textwidth}
%
%
\definecolor{mycolor1}{rgb}{0.63529,0.07843,0.18431}%
\definecolor{mycolor2}{rgb}{0.85098,0.32549,0.09804}%
\begin{tikzpicture}

\begin{axis}[%
width=0.974\fwidth,
height=\fheight,
at={(0\fwidth,0\fheight)},
scale only axis,
xmin=0,
xmax=50,
xlabel style={font=\color{white!15!black}},
xlabel={\small Reduced Dimension $2r$},
ymode=log,
ymin=7.4632155616569e-07,
ymax=1.2142180904913,
yminorticks=true,
ylabel style={font=\color{white!15!black}},
ylabel={\small Relative State Error},
axis background/.style={fill=white},
xmajorgrids,
ymajorgrids,
legend style={at={(0.119,0.735)}, anchor=south west, legend cell align=left, align=left, draw=white!15!black,font=\footnotesize}
]
\addplot [color=mycolor1, dashdotted, line width=1.5pt, mark=square, mark options={solid, mycolor1}]
  table[row sep=crcr]{%
2	1.21421809049131\\
4	0.067704312755749\\
6	0.00934690714697071\\
8	0.00287219777231627\\
10	0.00131262872099116\\
12	0.00135170885971506\\
14	0.0011490279219393\\
16	0.0011106508193067\\
18	0.000549900438108041\\
20	0.000360199156085349\\
22	0.000263209187077887\\
24	0.000236241321261613\\
26	0.000173437801307061\\
28	0.000139130295575117\\
30	9.71307543549341e-05\\
32	5.78866446175898e-05\\
34	3.90074249855312e-05\\
36	2.3538775737937e-05\\
38	1.52680717795679e-05\\
40	1.03084850799476e-05\\
42	7.24351429497894e-06\\
44	5.70496754186663e-06\\
46	4.97643786531497e-06\\
48	4.68237648605448e-06\\
50	4.57764661953644e-06\\
};
\addlegendentry{H-OpInf}

\addplot [color=mycolor2, dotted, line width=1.5pt, mark=x, mark options={solid, mycolor2}]
  table[row sep=crcr]{%
2	1.21421798799001\\
4	0.0676998094709061\\
6	0.00934474937172934\\
8	0.00287141800071053\\
10	0.00131231667958498\\
12	0.00135134031603344\\
14	0.00114874931293398\\
16	0.0011103624883497\\
18	0.000549740820173118\\
20	0.000360058926231183\\
22	0.000263065266981341\\
24	0.000236090410887765\\
26	0.000173286136274763\\
28	0.000138989280413002\\
30	9.70014445402571e-05\\
32	5.76810583052985e-05\\
34	3.87321147925901e-05\\
36	2.3095925614194e-05\\
38	1.45845609461908e-05\\
40	9.27218981496078e-06\\
42	5.67070300945859e-06\\
44	3.50725240375607e-06\\
46	2.12537290239771e-06\\
48	1.27081058362761e-06\\
50	7.46321556165683e-07\\
};
\addlegendentry{Intrusive PSD}

\end{axis}

\begin{axis}[%
width=1.257\fwidth,
height=1.257\fheight,
at={(-0.163\fwidth,-0.163\fheight)},
scale only axis,
xmin=0,
xmax=1,
ymin=0,
ymax=1,
axis line style={draw=none},
ticks=none,
axis x line*=bottom,
axis y line*=left
]
\end{axis}
\end{tikzpicture}%
\caption{State error (training data)}
 \label{fig:SG_state_train}
    \end{subfigure}
    \hspace{1.4cm}
    \begin{subfigure}{.38\textwidth}
           \setlength\fheight{4 cm}
           \setlength\fwidth{\textwidth}
%
%
\definecolor{mycolor1}{rgb}{0.63529,0.07843,0.18431}%
\definecolor{mycolor2}{rgb}{0.85098,0.32549,0.09804}%
\begin{tikzpicture}

\begin{axis}[%
width=0.974\fwidth,
height=\fheight,
at={(0\fwidth,0\fheight)},
scale only axis,
xmin=0,
xmax=50,
xlabel style={font=\color{white!15!black}},
xlabel={\small Reduced Dimension $2r$},
ymode=log,
ymin=0.001,
ymax=1.2840164638034,
yminorticks=true,
ylabel style={font=\color{white!15!black}},
ylabel={\small Relative State Error},
axis background/.style={fill=white},
xmajorgrids,
ymajorgrids,
legend style={at={(0.072,0.231)}, anchor=south west, legend cell align=left, align=left, draw=white!15!black,font=\footnotesize}
]
\addplot [color=mycolor1, dashdotted, line width=1.5pt, mark=square, mark options={solid, mycolor1}]
  table[row sep=crcr]{%
2	1.05471126970089\\
4	1.20150770882426\\
6	1.28400999028025\\
8	1.18979661325111\\
10	1.04082823183867\\
12	1.01765716915248\\
14	0.986222092810862\\
16	0.98515378192601\\
18	0.833090981226942\\
20	0.822737245354815\\
22	0.802343051631789\\
24	0.689755239166287\\
26	0.432716197540841\\
28	0.431111505008694\\
30	0.147695819307703\\
32	0.115234094884691\\
34	0.115818043788679\\
36	0.0692152397608271\\
38	0.0286465239497672\\
40	0.00822524606924847\\
42	0.00785302980640215\\
44	0.00600354644610799\\
46	0.00229598968131591\\
48	0.00215048873894658\\
50	0.0020217899073396\\
};
\addlegendentry{H-OpInf}

\addplot [color=mycolor2, dotted, line width=1.5pt, mark=x, mark options={solid, mycolor2}]
  table[row sep=crcr]{%
2	1.05471116555946\\
4	1.20151838744295\\
6	1.2840164638034\\
8	1.18981811892657\\
10	1.04086890494515\\
12	1.0177001314027\\
14	0.986267687744134\\
16	0.985199720962448\\
18	0.833168659280712\\
20	0.822815990136388\\
22	0.802426479604634\\
24	0.689872687482139\\
26	0.432984016034449\\
28	0.431381350969144\\
30	0.147912900008361\\
32	0.115466832905022\\
34	0.116054668067982\\
36	0.0695757998634243\\
38	0.0291256862908698\\
40	0.00872609609615491\\
42	0.00835806361842711\\
44	0.00652265664827633\\
46	0.00266013571968163\\
48	0.00251310077918586\\
50	0.00232442262829325\\
};
\addlegendentry{Intrusive PSD}

\end{axis}

\begin{axis}[%
width=1.257\fwidth,
height=1.257\fheight,
at={(-0.163\fwidth,-0.163\fheight)},
scale only axis,
xmin=0,
xmax=1,
ymin=0,
ymax=1,
axis line style={draw=none},
ticks=none,
axis x line*=bottom,
axis y line*=left
]
\end{axis}
\end{tikzpicture}%
\caption{State error (test data)}
 \label{fig:SG_state_test}
    \end{subfigure}
\caption{Sine-Gordon equation: Hamiltonian ROMs learned with H-OpInf achieve same accuracy as the intrusive Hamiltonian ROM for both training data and testing data, which is $400\%$ longer than the training time interval.}
 \label{fig:SG_state}
\end{figure}
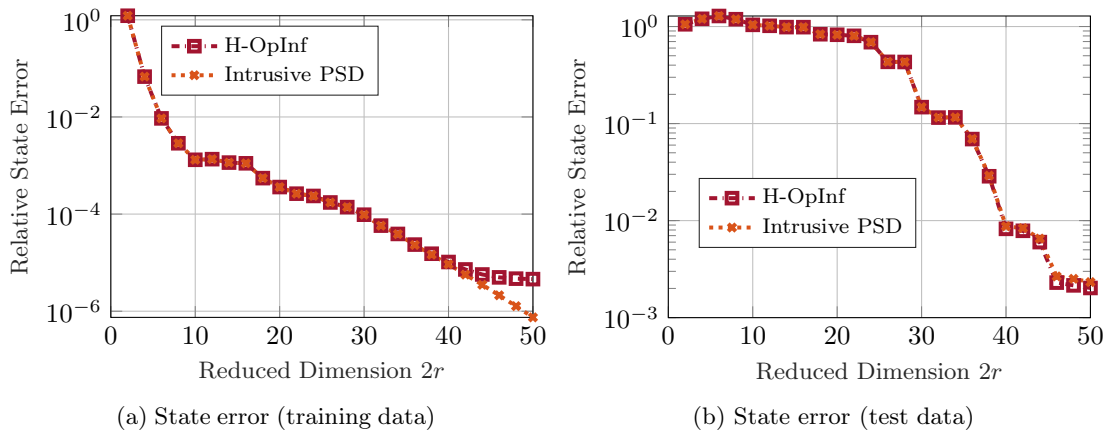
\begin{figure}
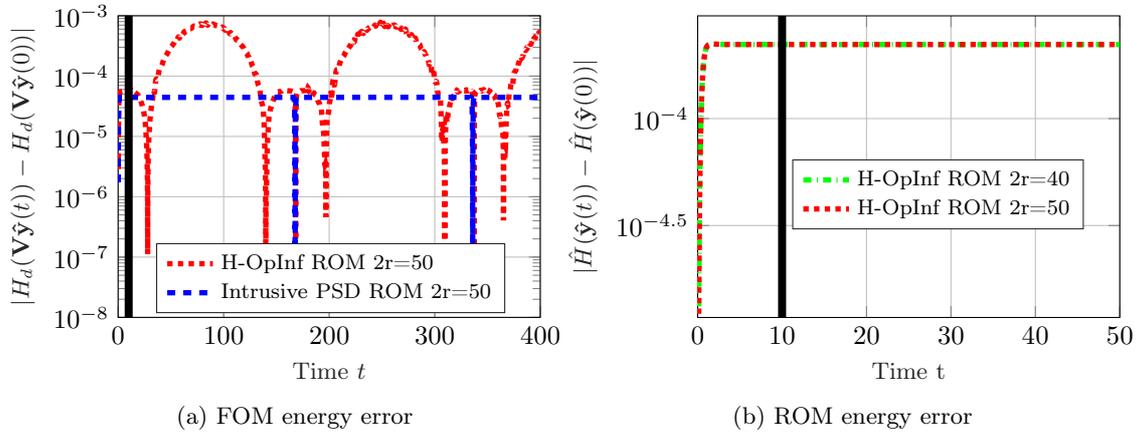

\captionsetup[subfigure]{oneside,margin={2cm,0 cm}}
\begin{subfigure}{.38\textwidth}
       \setlength\fheight{4 cm}
        \setlength\fwidth{\textwidth}
\input{Figures/figures/SG/SG_hfom.tex}
\caption{FOM energy error}
 \label{fig:SG_Hd}
    \end{subfigure}
    \hspace{1.4cm}
    \begin{subfigure}{.38\textwidth}
           \setlength\fheight{4 cm}
           \setlength\fwidth{\textwidth}
\input{Figures/figures/SG/SG_hrr.tex}
\caption{ROM energy error}
\label{fig:SG_Hr}
    \end{subfigure}
\caption{Sine-Gordon equation: Plot (a) shows that the learned Hamiltonian ROM exhibits bounded FOM energy error even at $t=400$. The black line indicates end of training time interval. Plot (b) shows similar ROM energy error behavior for learned Hamiltonian ROMs of different dimensions.}
 \label{fig:SG_energy}
\end{figure}
\begin{figure}
\captionsetup[subfigure]{oneside,margin={1.7cm,0 cm}}
\begin{subfigure}{.25\textwidth}
       \setlength\fheight{3 cm}
        \setlength\fwidth{\textwidth}
\raisebox{-51mm}{
%
%
\definecolor{mycolor1}{rgb}{0.00000,0.44700,0.74100}%
\definecolor{mycolor2}{rgb}{0.85000,0.32500,0.09800}%
\definecolor{mycolor3}{rgb}{0.92900,0.69400,0.12500}%
\begin{tikzpicture}

\begin{axis}[%
width=0.974\fwidth,
height=\fheight,
at={(0\fwidth,0\fheight)},
scale only axis,
xmin=-20,
xmax=20,
xlabel style={font=\color{white!15!black}},
xlabel={\small $x$},
ymin=0,
ymax=6,
ylabel style={font=\color{white!15!black}},
ylabel={\small $\varphi(x)$},
axis background/.style={fill=white},
xmajorgrids,
ymajorgrids,
legend style={at={(0.125,1)}, anchor=south west, legend cell align=left, align=left, draw=white!15!black,font=\footnotesize}
]
\addplot [color=mycolor1,line width=1.5pt]
  table[row sep=crcr]{%
-19.8	1.70753208745504e-07\\
-19.6	1.74093381720475e-07\\
-19.4	1.84354816427135e-07\\
-19.2	2.01847600682546e-07\\
-19	2.27349446642935e-07\\
-18.8	2.61830364363674e-07\\
-18.6	3.06666153096185e-07\\
-18.4	3.63706843178761e-07\\
-18.2	4.35112077786805e-07\\
-18	5.23887628601711e-07\\
-17.8	6.33451375274367e-07\\
-17.6	7.68228972010326e-07\\
-17.4	9.33696586534874e-07\\
-17.2	1.13632199685254e-06\\
-17	1.3843226185978e-06\\
-16.8	1.68770814247533e-06\\
-16.6	2.058479160577e-06\\
-16.4	2.51154672176379e-06\\
-16.2	3.06525825744291e-06\\
-16	3.74177336184423e-06\\
-15.8	4.56809798817862e-06\\
-15.6	5.5774800503072e-06\\
-15.4	6.81070398388833e-06\\
-15.2	8.31748647038344e-06\\
-15	1.01583552448861e-05\\
-14.8	1.24071286436688e-05\\
-14.6	1.5153969075563e-05\\
-14.4	1.85090511479778e-05\\
-14.2	2.26069903861104e-05\\
-14	2.76122358397449e-05\\
-13.8	3.37256597971321e-05\\
-13.6	4.11926135639444e-05\\
-13.4	5.03127716669973e-05\\
-13.2	6.14521578991412e-05\\
-13	7.50578348343516e-05\\
-12.8	9.16758459114083e-05\\
-12.6	0.000111973130002518\\
-12.4	0.000136764287911195\\
-12.2	0.000167044274985601\\
-12	0.000204028331848709\\
-11.8	0.000249200755685943\\
-11.6	0.000304374469236058\\
-11.4	0.000371763777802284\\
-11.2	0.000454073233557058\\
-11	0.000554606172263255\\
-10.8	0.000677397275978329\\
-10.6	0.000827374477587912\\
-10.4	0.00101055669702356\\
-10.2	0.00123429533054637\\
-10	0.00150756915856381\\
-9.8	0.00184134445970277\\
-9.6	0.00224901469649805\\
-9.4	0.0027469372600908\\
-9.2	0.00335508852697368\\
-9	0.004097862994216\\
-8.8	0.00500504761778624\\
-8.6	0.00611300874736747\\
-8.4	0.00746613622075927\\
-8.2	0.00911859709131081\\
-8	0.0111364596866393\\
-7.8	0.0136002563956781\\
-7.6	0.0166080593851688\\
-7.4	0.0202791456117553\\
-7.2	0.0247583249286874\\
-7	0.0302210005538736\\
-6.8	0.0368790389270604\\
-6.6	0.0449875862824489\\
-6.4	0.0548531684908256\\
-6.2	0.0668438922640929\\
-6	0.0814034863993909\\
-5.8	0.0990722720958023\\
-5.6	0.120519380261133\\
-5.4	0.146590125768218\\
-5.2	0.178368099118613\\
-5	0.21724199278012\\
-4.8	0.264955983528912\\
-4.6	0.323620488071168\\
-4.4	0.395679706600481\\
-4.2	0.483870213863625\\
-4	0.591227407954922\\
-3.8	0.72115911432281\\
-3.6	0.877508874660244\\
-3.4	1.06445192264611\\
-3.2	1.28608587282712\\
-3	1.54567239518629\\
-2.8	1.84456121006169\\
-2.6	2.18089921427123\\
-2.4	2.5484427767868\\
-2.2	2.93612186336968\\
-2	3.32903827016143\\
-1.8	3.71094534015714\\
-1.6	4.06728948103745\\
-1.4	4.3874807181004\\
-1.2	4.66562296346613\\
-1	4.89989296704102\\
-0.799999999999997	5.09126960194055\\
-0.599999999999998	5.24223558902565\\
-0.399999999999999	5.35576180751971\\
-0.199999999999999	5.43462969169474\\
0	5.48102834403995\\
0.199999999999999	5.49633796367476\\
0.399999999999999	5.48102834403995\\
0.599999999999998	5.43462969169474\\
0.799999999999997	5.35576180751971\\
1	5.24223558902565\\
1.2	5.09126960194055\\
1.4	4.89989296704102\\
1.6	4.66562296346613\\
1.8	4.3874807181004\\
2	4.06728948103744\\
2.2	3.71094534015714\\
2.4	3.32903827016142\\
2.6	2.93612186336968\\
2.8	2.54844277678681\\
3	2.18089921427123\\
3.2	1.84456121006169\\
3.4	1.54567239518629\\
3.6	1.28608587282712\\
3.8	1.06445192264611\\
4	0.877508874660242\\
4.2	0.721159114322808\\
4.4	0.591227407954921\\
4.6	0.483870213863624\\
4.8	0.395679706600481\\
5	0.323620488071168\\
5.2	0.264955983528912\\
5.4	0.217241992780121\\
5.6	0.178368099118614\\
5.8	0.146590125768218\\
6	0.120519380261134\\
6.2	0.0990722720958025\\
6.4	0.0814034863993911\\
6.6	0.066843892264093\\
6.8	0.0548531684908257\\
7	0.0449875862824489\\
7.2	0.0368790389270604\\
7.4	0.0302210005538736\\
7.6	0.0247583249286874\\
7.8	0.0202791456117553\\
8	0.0166080593851688\\
8.2	0.0136002563956781\\
8.4	0.0111364596866393\\
8.6	0.00911859709131077\\
8.8	0.00746613622075927\\
9	0.00611300874736748\\
9.2	0.00500504761778623\\
9.4	0.00409786299421599\\
9.6	0.00335508852697367\\
9.8	0.0027469372600908\\
10	0.00224901469649805\\
10.2	0.00184134445970277\\
10.4	0.00150756915856381\\
10.6	0.00123429533054637\\
10.8	0.00101055669702356\\
11	0.000827374477587913\\
11.2	0.000677397275978331\\
11.4	0.000554606172263253\\
11.6	0.000454073233557056\\
11.8	0.000371763777802284\\
12	0.000304374469236058\\
12.2	0.000249200755685943\\
12.4	0.000204028331848709\\
12.6	0.000167044274985601\\
12.8	0.000136764287911194\\
13	0.000111973130002518\\
13.2	9.16758459114081e-05\\
13.4	7.50578348343517e-05\\
13.6	6.14521578991413e-05\\
13.8	5.03127716669973e-05\\
14	4.11926135639445e-05\\
14.2	3.37256597971322e-05\\
14.4	2.7612235839745e-05\\
14.6	2.26069903861104e-05\\
14.8	1.85090511479778e-05\\
15	1.5153969075563e-05\\
15.2	1.24071286436688e-05\\
15.4	1.01583552448861e-05\\
15.6	8.31748647038343e-06\\
15.8	6.81070398388832e-06\\
16	5.5774800503072e-06\\
16.2	4.56809798817863e-06\\
16.4	3.74177336184423e-06\\
16.6	3.0652582574429e-06\\
16.8	2.51154672176378e-06\\
17	2.05847916057699e-06\\
17.2	1.68770814247534e-06\\
17.4	1.3843226185978e-06\\
17.6	1.13632199685254e-06\\
17.8	9.33696586534878e-07\\
18	7.68228972010327e-07\\
18.2	6.33451375274366e-07\\
18.4	5.23887628601709e-07\\
18.6	4.35112077786803e-07\\
18.8	3.6370684317876e-07\\
19	3.06666153096185e-07\\
19.2	2.61830364363675e-07\\
19.4	2.27349446642935e-07\\
19.6	2.01847600682546e-07\\
19.8	1.84354816427135e-07\\
20	1.74093381720475e-07\\
};

\addplot [color=OliveGreen, only marks, line width=1.5pt, mark size=2pt,mark=o, mark options={solid, OliveGreen}]
  table[row sep=crcr]{%
-19.8	1.70742790072206e-07\\
-17.8	6.33445294079258e-07\\
-15.8	4.56809766669333e-06\\
-13.8	3.37256581985701e-05\\
-11.8	0.000249206635121189\\
-9.8	0.0018412942370206\\
-7.8	0.0136002507304948\\
-5.8	0.0990721214846562\\
-3.8	0.721158932212248\\
-1.8	3.71094545379253\\
0.199999999999999	5.49633779062962\\
2.2	3.71094545379253\\
4.2	0.721158932212245\\
6.2	0.0990721214846565\\
8.2	0.0136002507304948\\
10.2	0.0018412942370206\\
12.2	0.000249206635121189\\
14.2	3.37256581985702e-05\\
16.2	4.56809766669334e-06\\
18.2	6.33445294102007e-07\\
};

\addplot [color=magenta, only marks,line width=1.5pt, mark size=3pt, mark=x, mark options={solid, magenta}]
  table[row sep=crcr]{%
-19.8	1.70743588712914e-07\\
-16.8	1.68768238225752e-06\\
-13.8	3.37255220404135e-05\\
-10.8	0.000677385447597146\\
-7.8	0.0136002385302641\\
-4.8	0.264954229862906\\
-1.8	3.71093386729692\\
1.2	5.09126096940336\\
4.2	0.721155389184079\\
7.2	0.0368789858879742\\
10.2	0.0018412754781499\\
13.2	9.16759934800141e-05\\
16.2	4.56808062827489e-06\\
19.2	2.61827666873045e-07\\
};

\end{axis}

\begin{axis}[%
width=1.257\fwidth,
height=1.257\fheight,
at={(-0.163\fwidth,-0.163\fheight)},
scale only axis,
xmin=0,
xmax=1,
ymin=0,
ymax=1,
axis line style={draw=none},
ticks=none,
axis x line*=bottom,
axis y line*=left
]
\end{axis}
\end{tikzpicture}%
}
\subcaption{$t=5$}
    \end{subfigure}
    \hspace{.8cm}
    \begin{subfigure}{.25\textwidth}
           \setlength\fheight{3 cm}
           \setlength\fwidth{\textwidth}
%
%
\definecolor{mycolor1}{rgb}{0.00000,0.44700,0.74100}%
\definecolor{mycolor2}{rgb}{0.85000,0.32500,0.09800}%
\definecolor{mycolor3}{rgb}{0.92900,0.69400,0.12500}%
\begin{tikzpicture}

\begin{axis}[%
width=0.974\fwidth,
height=\fheight,
at={(0\fwidth,0\fheight)},
scale only axis,
xmin=-20,
xmax=20,
xlabel style={font=\color{white!15!black}},
xlabel={\small $x$},
ymin=-1,
ymax=7,
ylabel style={font=\color{white!15!black}},
ylabel={\small $\varphi(x)$},
axis background/.style={fill=white},
xmajorgrids,
ymajorgrids,
legend style={at={(-0.125,1)}, anchor=south west, legend cell align=left, align=left, draw=white!15!black,font=\footnotesize}
]
\addplot [color=mycolor1,line width=1.5pt]
  table[row sep=crcr]{%
-19.8	-0.00323910529384879\\
-19.6	-0.0031702059776167\\
-19.4	-0.00294159250474462\\
-19.2	-0.00252238863851161\\
-19	-0.00192810657941119\\
-18.8	-0.00121665432312324\\
-18.6	-0.000426340393514913\\
-18.4	0.000458824101624434\\
-18.2	0.00145956828510256\\
-18	0.00251601232509763\\
-17.8	0.00346140457495827\\
-17.6	0.00408574789140677\\
-17.4	0.00423596784460692\\
-17.2	0.00388640088455403\\
-17	0.00315135417435002\\
-16.8	0.00223983017279993\\
-16.6	0.00137068982724359\\
-16.4	0.000691998773082665\\
-16.2	0.000249893077042307\\
-16	9.76753782084734e-06\\
-15.8	-0.000103860778994342\\
-15.6	-0.000172248761922204\\
-15.4	-0.000255348366713753\\
-15.2	-0.000379215451895256\\
-15	-0.000542920638435058\\
-14.8	-0.000732267169240544\\
-14.6	-0.000927678879530299\\
-14.4	-0.00110997529579728\\
-14.2	-0.00126624831277172\\
-14	-0.00138979309805457\\
-13.8	-0.00147668161416411\\
-13.6	-0.00152556205640948\\
-13.4	-0.00153722526174\\
-13.2	-0.00151187079306181\\
-13	-0.00144895729410873\\
-12.8	-0.00134815757670074\\
-12.6	-0.00120798025273974\\
-12.4	-0.00102554735485453\\
-12.2	-0.000797563182507789\\
-12	-0.000519279917292474\\
-11.8	-0.000184014953346825\\
-11.6	0.000216276714593358\\
-11.4	0.000692061036744088\\
-11.2	0.00125687473062062\\
-11	0.00192723831327041\\
-10.8	0.00272397210995032\\
-10.6	0.00367308017594523\\
-10.4	0.00480636222655143\\
-10.2	0.00616327670965007\\
-10	0.00779236357353073\\
-9.8	0.00975298969666313\\
-9.6	0.0121181972119743\\
-9.4	0.0149773156207253\\
-9.2	0.018439589921711\\
-9	0.0226388073881353\\
-8.8	0.0277383191857163\\
-8.6	0.0339378091802928\\
-8.4	0.0414812485032922\\
-8.2	0.0506664510587968\\
-8	0.0618571707551973\\
-7.8	0.075497316799274\\
-7.6	0.0921285356239492\\
-7.4	0.112411492686997\\
-7.2	0.137151297450447\\
-7	0.16732835790199\\
-6.8	0.204134680908963\\
-6.6	0.249016518662813\\
-6.4	0.303723083207673\\
-6.2	0.370359810195845\\
-6	0.451442891977579\\
-5.8	0.549946628543931\\
-5.6	0.669328531239194\\
-5.4	0.813504282346997\\
-5.2	0.986726670145177\\
-5	1.19329974070512\\
-4.8	1.43703977231023\\
-4.6	1.72040479423681\\
-4.4	2.04329933645865\\
-4.2	2.40177468910018\\
-4	2.78715911377525\\
-3.8	3.18634209176476\\
-3.6	3.58362390906713\\
-3.4	3.96366031499554\\
-3.2	4.31424841586852\\
-3	4.627819804219\\
-2.8	4.90140020989993\\
-2.6	5.13557741481546\\
-2.4	5.33318926774846\\
-2.2	5.4981925352415\\
-2	5.63487152548093\\
-1.8	5.74736361735484\\
-1.6	5.83941642922848\\
-1.4	5.91429130218909\\
-1.2	5.97474975156477\\
-1	6.02308138726547\\
-0.799999999999997	6.06114894973793\\
-0.599999999999998	6.09043673758778\\
-0.399999999999999	6.112095533214\\
-0.199999999999999	6.1269806409093\\
0	6.13568175207698\\
0.199999999999999	6.13854421888854\\
0.399999999999999	6.13568175207698\\
0.599999999999998	6.1269806409093\\
0.799999999999997	6.112095533214\\
1	6.09043673758778\\
1.2	6.06114894973793\\
1.4	6.02308138726547\\
1.6	5.97474975156476\\
1.8	5.91429130218909\\
2	5.83941642922848\\
2.2	5.74736361735483\\
2.4	5.63487152548094\\
2.6	5.4981925352415\\
2.8	5.33318926774845\\
3	5.13557741481545\\
3.2	4.90140020989991\\
3.4	4.62781980421899\\
3.6	4.31424841586851\\
3.8	3.96366031499552\\
4	3.5836239090671\\
4.2	3.18634209176475\\
4.4	2.78715911377524\\
4.6	2.40177468910017\\
4.8	2.04329933645864\\
5	1.72040479423681\\
5.2	1.43703977231022\\
5.4	1.19329974070511\\
5.6	0.986726670145172\\
5.8	0.813504282346991\\
6	0.669328531239186\\
6.2	0.54994662854393\\
6.4	0.451442891977584\\
6.6	0.370359810195848\\
6.8	0.303723083207675\\
7	0.249016518662802\\
7.2	0.204134680908955\\
7.4	0.167328357902005\\
7.6	0.137151297450446\\
7.8	0.112411492686984\\
8	0.0921285356239502\\
8.2	0.0754973167992724\\
8.4	0.0618571707552013\\
8.6	0.0506664510588051\\
8.8	0.0414812485032899\\
9	0.0339378091802868\\
9.2	0.0277383191857161\\
9.4	0.0226388073881379\\
9.6	0.018439589921707\\
9.8	0.0149773156207208\\
10	0.012118197211969\\
10.2	0.00975298969665298\\
10.4	0.00779236357352889\\
10.6	0.00616327670965384\\
10.8	0.00480636222655385\\
11	0.0036730801759532\\
11.2	0.00272397210995947\\
11.4	0.00192723831327387\\
11.6	0.00125687473061928\\
11.8	0.000692061036743197\\
12	0.000216276714597377\\
12.2	-0.000184014953344122\\
12.4	-0.000519279917292006\\
12.6	-0.000797563182501592\\
12.8	-0.00102554735484979\\
13	-0.00120798025274787\\
13.2	-0.00134815757671349\\
13.4	-0.00144895729411381\\
13.6	-0.00151187079306201\\
13.8	-0.00153722526174123\\
14	-0.00152556205641006\\
14.2	-0.00147668161416322\\
14.4	-0.00138979309805562\\
14.6	-0.00126624831277433\\
14.8	-0.00110997529579646\\
15	-0.00092767887952601\\
15.2	-0.000732267169239371\\
15.4	-0.000542920638439288\\
15.6	-0.000379215451898943\\
15.8	-0.000255348366712739\\
16	-0.000172248761919957\\
16.2	-0.000103860778994952\\
16.4	9.76753781983323e-06\\
16.6	0.000249893077046666\\
16.8	0.000691998773093052\\
17	0.00137068982725319\\
17.2	0.00223983017280158\\
17.4	0.00315135417434416\\
17.6	0.0038864008845471\\
17.8	0.00423596784460329\\
18	0.00408574789140572\\
18.2	0.00346140457495796\\
18.4	0.0025160123250978\\
18.6	0.00145956828510282\\
18.8	0.000458824101623423\\
19	-0.000426340393517491\\
19.2	-0.00121665432312539\\
19.4	-0.00192810657941064\\
19.6	-0.00252238863850823\\
19.8	-0.00294159250474037\\
20	-0.0031702059776139\\
};
\addlegendentry{FOM}

\addplot [color=OliveGreen, only marks, line width=1.5pt, mark size=2pt,mark=o, mark options={solid, OliveGreen}]
  table[row sep=crcr]{%
-19.8	1.09454920250097e-06\\
-17.8	3.49571390404372e-06\\
-15.8	2.36026221444817e-05\\
-13.8	0.00017400672156897\\
-11.8	0.000897853805377603\\
-9.8	0.0061647866292322\\
-7.8	0.0759432652372956\\
-5.8	0.551428561551002\\
-3.8	3.1844950306646\\
-1.8	5.74859284303496\\
0.199999999999999	6.14090176709627\\
2.2	5.74859284303578\\
4.2	3.18449503066508\\
6.2	0.551428561551035\\
8.2	0.0759432652372931\\
10.2	0.00616478662923254\\
12.2	0.000897853805378052\\
14.2	0.000174006721568991\\
16.2	2.36026221444864e-05\\
18.2	3.49571390790414e-06\\
};
\addlegendentry{Intrusive PSD ROM 2r=50}

\addplot [color=magenta, only marks,line width=1.5pt, mark size=3pt, mark=x, mark options={solid, magenta}]
  table[row sep=crcr]{%
-19.8	1.09356116054735e-06\\
-16.8	8.51494015833499e-06\\
-13.8	0.000173906383638138\\
-10.8	0.000606367250067395\\
-7.8	0.0759356308656905\\
-4.8	1.43750134443467\\
-1.8	5.74824455570553\\
1.2	6.06142314728673\\
4.2	3.18455775495952\\
7.2	0.205433435257801\\
10.2	0.00593856924411699\\
13.2	0.000451417002910482\\
16.2	2.35811847171332e-05\\
19.2	1.55689548161383e-06\\
};
\addlegendentry{H-OpInf ROM 2r=50}

\end{axis}
\end{tikzpicture}%
\subcaption{$t=25$}
    \end{subfigure}
    \hspace{.8cm}
        \begin{subfigure}{.25\textwidth}
           \setlength\fheight{3 cm}
           \setlength\fwidth{\textwidth}
\raisebox{-51mm}{
%
%
\definecolor{mycolor1}{rgb}{0.00000,0.44700,0.74100}%
\definecolor{mycolor2}{rgb}{0.85000,0.32500,0.09800}%
\definecolor{mycolor3}{rgb}{0.92900,0.69400,0.12500}%
\begin{tikzpicture}

\begin{axis}[%
width=0.974\fwidth,
height=\fheight,
at={(0\fwidth,0\fheight)},
scale only axis,
xmin=-20,
xmax=20,
xlabel style={font=\color{white!15!black}},
xlabel={\small $x$},
ymin=-1,
ymax=7,
ylabel style={font=\color{white!15!black}},
ylabel={\small $\varphi(x)$},
axis background/.style={fill=white},
xmajorgrids,
ymajorgrids,
legend style={at={(0.125,1)}, anchor=south west, legend cell align=left, align=left, draw=white!15!black,font=\footnotesize}
]
\addplot [color=mycolor1,line width=1.5pt]
  table[row sep=crcr]{%
-19.8	-0.00165536331981922\\
-19.6	-0.00164902868838103\\
-19.4	-0.00163022520715252\\
-19.2	-0.00159908962526114\\
-19	-0.00155584296575048\\
-18.8	-0.00150105298758812\\
-18.6	-0.00143513693313169\\
-18.4	-0.00135873953116342\\
-18.2	-0.0012727185151947\\
-18	-0.0011777800733211\\
-17.8	-0.00107492777584711\\
-17.6	-0.000965232351822879\\
-17.4	-0.000849660001894636\\
-17.2	-0.000729450929945928\\
-17	-0.000605814001082605\\
-16.8	-0.000479869441452028\\
-16.6	-0.000352937214114831\\
-16.4	-0.000226255683764191\\
-16.2	-0.00010089869039601\\
-16	2.19118179305353e-05\\
-15.8	0.000141108955256335\\
-15.6	0.000255946956353476\\
-15.4	0.000365537693567874\\
-15.2	0.000469324950874439\\
-15	0.000567245502813087\\
-14.8	0.000659127805931408\\
-14.6	0.000745413116532946\\
-14.4	0.000827209337904145\\
-14.2	0.000905692892206006\\
-14	0.000982966239973399\\
-13.8	0.00106209271295507\\
-13.6	0.0011465709092857\\
-13.4	0.00124116305203549\\
-13.2	0.00135219573905364\\
-13	0.00148699391431375\\
-12.8	0.00165468465930827\\
-12.6	0.00186700221799081\\
-12.4	0.00213762113977289\\
-12.2	0.00248305451779483\\
-12	0.00292405014900385\\
-11.8	0.00348513024815982\\
-11.6	0.0041955560581957\\
-11.4	0.00509148989828097\\
-11.2	0.00621638365347517\\
-11	0.00762195228323158\\
-10.8	0.00937150398153406\\
-10.6	0.0115420205267063\\
-10.4	0.0142256753379807\\
-10.2	0.0175347439962698\\
-10	0.0216066497249734\\
-9.8	0.0266075798420703\\
-9.6	0.0327394005865445\\
-9.4	0.0402492963337323\\
-9.2	0.0494378988382979\\
-9	0.0606693280859482\\
-8.8	0.0743869146158208\\
-8.6	0.0911298877916469\\
-8.4	0.111551234080324\\
-8.2	0.136443997650077\\
-8	0.166775719119393\\
-7.8	0.203726519758962\\
-7.6	0.248735358600837\\
-7.4	0.30355779887011\\
-7.2	0.370324838522379\\
-7	0.451592813575349\\
-6.8	0.550380915178283\\
-6.6	0.670181020790942\\
-6.4	0.814907421598943\\
-6.2	0.9887511295477\\
-6	1.19589153592744\\
-5.8	1.4399849673014\\
-5.6	1.72334533009277\\
-5.4	2.04581700595455\\
-5.2	2.40353856340058\\
-5	2.78808170461768\\
-4.8	3.1866520229053\\
-4.6	3.58379654076675\\
-4.4	3.96423987307571\\
-4.2	4.31566325357786\\
-4	4.6302864945292\\
-3.8	4.90495131386337\\
-3.6	5.14017485634973\\
-3.4	5.33883816536122\\
-3.2	5.50498620581662\\
-3	5.64296818658334\\
-2.8	5.75695415105117\\
-2.6	5.85073440315998\\
-2.4	5.92766427768708\\
-2.2	5.99065508503627\\
-2	6.04217611796919\\
-1.8	6.08426906063967\\
-1.6	6.1185843851988\\
-1.4	6.14644497733406\\
-1.2	6.16892605799016\\
-1	6.18691856438123\\
-0.799999999999997	6.20114842346642\\
-0.599999999999998	6.21216637636094\\
-0.399999999999999	6.22035679758509\\
-0.199999999999999	6.22598766473813\\
0	6.22926461576766\\
0.199999999999999	6.23033764696073\\
0.399999999999999	6.22926461576767\\
0.599999999999998	6.22598766473815\\
0.799999999999997	6.22035679758509\\
1	6.21216637636095\\
1.2	6.20114842346641\\
1.4	6.18691856438123\\
1.6	6.16892605799018\\
1.8	6.14644497733403\\
2	6.1185843851988\\
2.2	6.08426906063965\\
2.4	6.04217611796921\\
2.6	5.99065508503625\\
2.8	5.92766427768709\\
3	5.85073440315997\\
3.2	5.75695415105119\\
3.4	5.64296818658331\\
3.6	5.50498620581661\\
3.8	5.33883816536121\\
4	5.14017485634971\\
4.2	4.90495131386334\\
4.4	4.63028649452918\\
4.6	4.31566325357782\\
4.8	3.96423987307567\\
5	3.58379654076673\\
5.2	3.1866520229053\\
5.4	2.78808170461765\\
5.6	2.40353856340057\\
5.8	2.04581700595451\\
6	1.72334533009275\\
6.2	1.43998496730139\\
6.4	1.19589153592742\\
6.6	0.988751129547713\\
6.8	0.814907421598943\\
7	0.67018102079093\\
7.2	0.550380915178271\\
7.4	0.451592813575354\\
7.6	0.37032483852236\\
7.8	0.303557798870109\\
8	0.248735358600845\\
8.2	0.203726519758959\\
8.4	0.166775719119411\\
8.6	0.136443997650078\\
8.8	0.111551234080315\\
9	0.0911298877916391\\
9.2	0.0743869146158108\\
9.4	0.0606693280859316\\
9.6	0.0494378988382881\\
9.8	0.0402492963337239\\
10	0.0327394005865431\\
10.2	0.0266075798420849\\
10.4	0.0216066497249716\\
10.6	0.0175347439962628\\
10.8	0.0142256753379758\\
11	0.0115420205267127\\
11.2	0.00937150398154131\\
11.4	0.00762195228322014\\
11.6	0.00621638365348012\\
11.8	0.00509148989829034\\
12	0.00419555605819097\\
12.2	0.00348513024814717\\
12.4	0.00292405014899412\\
12.6	0.00248305451780117\\
12.8	0.00213762113978095\\
13	0.0018670022180038\\
13.2	0.00165468465931374\\
13.4	0.00148699391430207\\
13.6	0.00135219573905753\\
13.8	0.00124116305204373\\
14	0.00114657090927348\\
14.2	0.00106209271295189\\
14.4	0.00098296623998944\\
14.6	0.000905692892218253\\
14.8	0.000827209337914731\\
15	0.000745413116536008\\
15.2	0.000659127805923094\\
15.4	0.000567245502808214\\
15.6	0.000469324950877682\\
15.8	0.000365537693582893\\
16	0.000255946956358104\\
16.2	0.000141108955246509\\
16.4	2.19118179301561e-05\\
16.6	-0.000100898690404277\\
16.8	-0.000226255683772003\\
17	-0.000352937214103135\\
17.2	-0.000479869441449329\\
17.4	-0.000605814001091448\\
17.6	-0.000729450929947111\\
17.8	-0.000849660001891513\\
18	-0.000965232351833939\\
18.2	-0.00107492777586598\\
18.4	-0.00117778007331535\\
18.6	-0.00127271851518419\\
18.8	-0.00135873953117156\\
19	-0.00143513693313013\\
19.2	-0.00150105298758574\\
19.4	-0.00155584296576614\\
19.6	-0.00159908962527108\\
19.8	-0.00163022520715947\\
20	-0.00164902868839655\\
};

\addplot [color=OliveGreen, only marks, line width=1.5pt, mark size=2pt,mark=o, mark options={solid, OliveGreen}]
  table[row sep=crcr]{%
-19.8	3.11592461593548e-06\\
-17.8	8.2669857117928e-06\\
-15.8	4.93271096781639e-05\\
-13.8	0.000364227495978092\\
-11.8	0.00107280755960499\\
-9.8	0.00597910897746273\\
-7.8	0.192005301023815\\
-5.8	1.40538899429557\\
-3.8	4.87767048610249\\
-1.8	6.0765581329859\\
0.199999999999999	6.22841468771902\\
2.2	6.07655813298987\\
4.2	4.87767048610421\\
6.2	1.4053889942958\\
8.2	0.192005301023795\\
10.2	0.00597910897746051\\
12.2	0.00107280755960678\\
14.2	0.000364227495978186\\
16.2	4.93271096782045e-05\\
18.2	8.26698571276825e-06\\
};

\addplot [color=magenta, only marks,line width=1.5pt, mark size=3pt, mark=x, mark options={solid, magenta}]
  table[row sep=crcr]{%
-19.8	3.13201077093607e-06\\
-16.8	1.81916504660277e-05\\
-13.8	0.000365481187745102\\
-10.8	-0.0049824160229118\\
-7.8	0.19332280593339\\
-4.8	3.15516205338505\\
-1.8	6.07816776148356\\
1.2	6.19637036470351\\
4.2	4.88571562596532\\
7.2	0.535787001097117\\
10.2	0.00513095449253071\\
13.2	0.000903673299683858\\
16.2	4.947044364196e-05\\
19.2	4.13035979465337e-06\\
};

\end{axis}
\end{tikzpicture}%
}
\subcaption{$t=50$}
    \end{subfigure}
    \caption{Sine-Gordon equation: Plots show the numerical approximation of the solution of \eqref{eq:sg} using low-dimensional ($2r=50$) intrusive and nonintrusive Hamiltonian ROM at different $t$ values.}
\label{fig:SG_traj}
\end{figure}
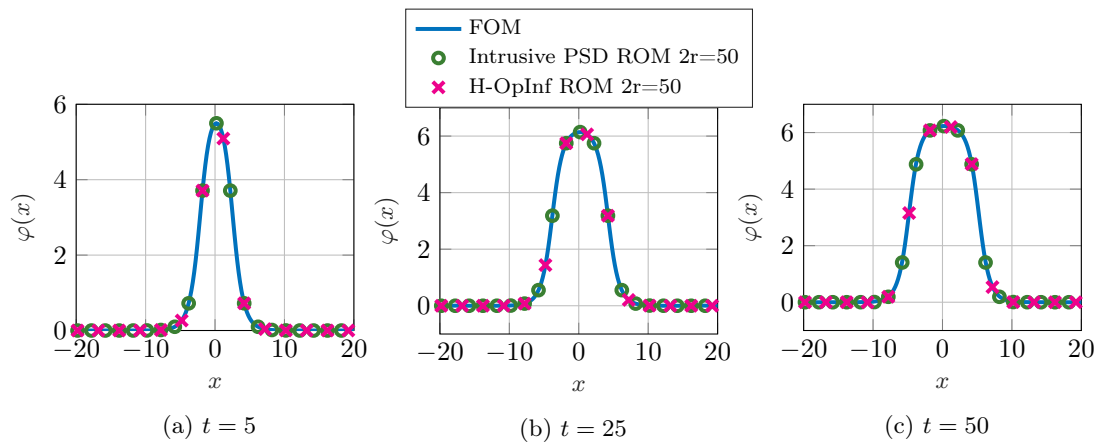
%
\section{Conclusions} 
\label{s:5} 
We have presented a data-driven model reduction method that utilizes information about the space-time continuous Hamiltonian functional to derive Hamiltonian ROMs via nonintrusive operator inference. Our method applies to canonical Hamiltonian systems with nonpolynomial nonlinear terms, and learns Hamiltonian reduced operators directly from the full-model simulation data, \textit{without} an intrusive symplectic projection step that requires full-model operators. Our method only requires access to the \textit{form} of the space-discretized Hamiltonian so that we can learn the parameters, but not the space-discretized Hamiltonian itself.  The inference of the operators is based on a constrained least-squares problem that ensures that the reduced models are Hamiltonian systems. We have also presented a theoretical result that shows that the nonintrusively learned reduced operators converge to the same reduced operators as obtained with intrusive structure-preserving model reduction under certain not-too-restrictive conditions.
\par
The numerical experiments with the nonlinear Schr\"odinger equation and the sine-Gordon equation demonstrate that our method works well for Hamiltonian systems with complex nonlinear phenomena. The numerical results also show that the presented method learns stable reduced-order models and provides greater interpretability, while facilitating accurate long-time predictions far outside the training data regime.
\par 
Future research directions motivated by this work are: extending Hamiltonian operator inference to noncanonical Hamiltonian systems; deriving error bounds for the difference between Hamiltonian of intrusive and nonintrusive ROM, i.e., $|\tilde{H}-\hat{H}|$; and extending sampling algorithm based on re-projection for implicit time-marching schemes so that they can be combined with H-OpInf to recover intrusive Hamiltonian ROMs in a nonintrusive way.

\paragraph{Acknowledgments:} The authors would like to thank Silke Glas and Patrick Buchfink for valuable feedback, as well as the anonymous reviewers for their comments which helped improve the paper. Z.W. was partially supported by the National Science Foundation under grant DMS-1913073. B.K. was partially supported by the National Science Foundation under grant PHY-2028125.

\bibliographystyle{vancouver}
\bibliography{main}
\end{document}